\documentclass[11pt]{article}
\usepackage{latexsym}
\usepackage{hyperref}
\usepackage[all]{xy}

\usepackage{amsmath}
\usepackage{xypic}

\usepackage{amsmath}
\usepackage{amssymb}
\usepackage{amsfonts}
\usepackage{amsthm}
\usepackage{mathrsfs}
\usepackage{stmaryrd}
\usepackage[utf8]{inputenc}
\usepackage{xcolor}
\usepackage{newunicodechar}
\usepackage{xspace}

\usepackage{enumitem}
\setlist{nosep} 

\usepackage[lmargin=1.5in,rmargin=1.5in,tmargin=1.5in,bmargin=1.5in]{geometry}

\usepackage{multicol}

\DeclareUnicodeCharacter{00A0}{ }

\newunicodechar{α}{\ensuremath{\alpha}}
\newunicodechar{β}{\ensuremath{\beta}}
\newunicodechar{χ}{\ensuremath{\chi}}
\newunicodechar{δ}{\ensuremath{\delta}}
\newunicodechar{ε}{\ensuremath{\varepsilon}}
\newunicodechar{Δ}{\ensuremath{\Delta}}
\newunicodechar{η}{\ensuremath{\eta}}
\newunicodechar{γ}{\ensuremath{\gamma}}
\newunicodechar{Γ}{\ensuremath{\Gamma}}
\newunicodechar{ι}{\ensuremath{\iota}}
\newunicodechar{κ}{\ensuremath{\kappa}}
\newunicodechar{λ}{\ensuremath{\lambda}}
\newunicodechar{Λ}{\ensuremath{\Lambda}}
\newunicodechar{μ}{\ensuremath{\mu}}
\newunicodechar{ω}{\ensuremath{\omega}}
\newunicodechar{Ω}{\ensuremath{\Omega}}
\newunicodechar{π}{\ensuremath{\pi}}
\newunicodechar{φ}{\ensuremath{\phi}}
\newunicodechar{Φ}{\ensuremath{\Phi}}
\newunicodechar{ψ}{\ensuremath{\psi}}
\newunicodechar{Ψ}{\ensuremath{\Psi}}
\newunicodechar{ρ}{\ensuremath{\rho}}
\newunicodechar{σ}{\ensuremath{\sigma}}
\newunicodechar{Σ}{\ensuremath{\Sigma}}
\newunicodechar{τ}{\ensuremath{\tau}}
\newunicodechar{θ}{\ensuremath{\theta}}
\newunicodechar{Θ}{\ensuremath{\Theta}}
\newunicodechar{ζ}{\ensuremath{\zeta}}
\newunicodechar{ν}{\ensuremath{\nu}}

\newunicodechar{∞}{\ensuremath{\infty}}
\newunicodechar{→}{\ensuremath{\to}}
\newunicodechar{←}{\ensuremath{\leftarrow}}
\newunicodechar{×}{\ensuremath{\times}}
\newunicodechar{∪}{\ensuremath{\cup}}
\newunicodechar{∩}{\ensuremath{\cap}}
\newunicodechar{⊇}{\ensuremath{\supseteq}}
\newunicodechar{⊃}{\ensuremath{\supset}}
\newunicodechar{⊆}{\ensuremath{\subseteq}}
\newunicodechar{⊂}{\ensuremath{\subset}}
\newunicodechar{∈}{\ensuremath{\in}}
\newunicodechar{○}{\ensuremath{\circ}}

\DeclareFontFamily{U}{wncy}{}
\DeclareFontShape{U}{wncy}{m}{n}{<->wncyr10}{}
\DeclareSymbolFont{mcy}{U}{wncy}{m}{n}
\DeclareMathSymbol{\Sh}{\mathord}{mcy}{"58}

\renewcommand{\implies}{\Rightarrow}

\newtheorem{theo}{Theorem}[section]
\newtheorem{prop}[theo]{Proposition}
\newtheorem{coro}[theo]{Corollary}
\newtheorem{lemm}[theo]{Lemma}

\theoremstyle{definition}
\newtheorem{defi}[theo]{Definition}

\newtheorem{exam}[theo]{Example}
\newtheorem{coun}[theo]{Counterexample}
\newtheorem{rema}[theo]{Remark}

\newtheorem{rede}[theo]{Remark/Definition}

\newtheorem{conj}[theo]{Conjecture}

\newtheorem{warn}[theo]{Warning}

\DeclareMathOperator{\uSpec}{\underline{Spec}}
\newcommand{\cA}{\mathcal{A}}
\renewcommand{\AA}{\mathbb{A}}

\newcommand{\cC}{\mathcal{C}}

\newcommand{\cE}{\mathcal{E}}
\newcommand{\FF}{\mathbb{F}}
\newcommand{\GG}{\mathbb{G}}

\newcommand{\cK}{\mathcal{K}}

\newcommand{\m}{\mathfrak{m}}
\newcommand{\cM}{\mathcal{M}}

\newcommand{\sU}{\mathcal{U}}
\newcommand{\n}{\mathfrak{n}}
\newcommand{\NN}{\mathbb{N}}

\newcommand{\OO}{\mathcal{O}}
\newcommand{\p}{\mathfrak{p}}

\newcommand{\QQ}{\mathbb{Q}}

\newcommand{\cS}{\mathcal{S}}

\newcommand{\XX}{\mathbb{X}}
\newcommand{\cY}{\mathcal{Y}}
\newcommand{\cZ}{\mathcal{Z}}
\newcommand{\ZZ}{\mathbb{Z}}

\DeclareMathOperator{\Ab}{Ab}

\DeclareMathOperator{\Spec}{Spec}

\DeclareMathOperator{\id}{id}

\DeclareMathOperator{\Map}{Map}
\DeclareMathOperator{\fib}{fib}

\DeclareMathOperator{\length}{length}

\newcommand{\fN}{\frak{N}}

\newcommand{\Set}{\mathrm{Set}}
\newcommand{\Spt}{\mathrm{Spt}}

\newcommand{\PSh}{\mathrm{PSh}}

\newcommand{\Shv}{\mathrm{Shv}}

\newcommand{\Sm}{\mathrm{Sm}}

\newcommand{\Sch}{\mathrm{Sch}}

\DeclareMathOperator{\RZ}{RZ}

\newcommand{\Ind}{\mathrm{Ind}}

\DeclareMathOperator{\Zar}{Zar}
\DeclareMathOperator{\Nis}{Nis}
\DeclareMathOperator{\car}{car}

\newcommand{\cdh}{\mathrm{cdh}}
\newcommand{\procdh}{\mathrm{procdh}}

\DeclareMathOperator{\Arr}{Arr}
\newcommand{\Fun}{\mathrm{Fun}}

\DeclareMathOperator{\Mod}{Mod}

\newcommand{\Frac}{\mathrm{Frac}}

\newcommand{\tr}{{\operatorname{tr}}}

\newcommand{\qfor}{\quad\text{for }}

\newcommand{\colim}{\operatorname{colim}}

\newcommand{\qcqs}{\mathrm{qcqs}}
\newcommand{\noe}{\mathrm{noe}}
\newcommand{\gen}{\mathrm{gen}}

\newcommand{\op}{\mathrm{op}}

\newcommand{\red}{\mathrm{red}}

\def\sI{\mathcal{I}}

\newcommand{\Nil}{\operatorname{Nil}}

\def\conc{\tau_{\geq 0}}

\def\tK{\widetilde{K}}

\newcommand{\Fmot}{Fil_{\operatorname{mot}}}
\newcommand{\Znsm}{\ZZ(n)^{\operatorname{sm}}}
\newcommand{\Znem}{\ZZ(n)^{\operatorname{EM}}}
\newcommand{\Zncdh}{\ZZ(n)^{\cdh}}
\newcommand{\Znpcdh}{\ZZ(n)^{\procdh}}
\newcommand{\Lsm}{L^{\operatorname{sm}}}
\newcommand{\Fpcdh}{Fil_{\procdh}}

\newcommand{\HN}{\operatorname{HC}^-}
\newcommand{\TC}{\operatorname{TC}}
\newcommand{\HKR}{\operatorname{HKR}}
\newcommand{\FHKR}{\ensuremath{Fil_{\HKR}}}
\newcommand{\BMS}{\operatorname{BMS}}
\newcommand{\FBMS}{\ensuremath{Fil_{\BMS}}}
\newcommand{\gr}{\operatorname{gr}}
\newcommand{\LO}[1]{L\Omega_{#1}}
\newcommand{\hLO}[1]{\widehat{L\Omega}_{#1}}
\newcommand{\hLOH}[2]{\widehat{L\Omega}^{\geq #1}_{#2}}
\newcommand{\Znsyn}{{\ZZ(n)^{\operatorname{syn}}}}

\newcommand{\conDescent}{Desc}
\newcommand{\conFinitary}{Fin}
\newcommand{\conBB}{BB}%

\newcounter{spec}
{\end{list}}%

\begin{document}
\title{A procdh topology}
\author{Shane Kelly and Shuji Saito}
\maketitle

\begin{abstract}
In this article we propose a definition of a procdh topos. We show that it encodes procdh excision, has bounded homotopy dimension and therefore is hypercomplete and admits a conservative family of fibre functors. We also describe the local rings.

As an application, we show that nonconnective $K$-theory is the procdh sheafification of connective $K$-theory, and that the motivic cohomology recently proposed by Elmanto and Morrow is the procdh sheafification of Voevodsky's motivic cohomology.
\end{abstract}

\tableofcontents

\section{Introduction}


It has been known for many decades that blowups of smooth varieties in smooth centres give rise to a long exact sequence in algebraic $K$-theory \cite[Exp.VII]{SGA6}, \cite{TT90}. 
As part of his work on the Bloch-Kato conjecture, Voevodsky encoded such blowup exact sequences in a Grothendieck topology---the cdh topology---giving access to the full arsenal of topos theory, \cite{SVBK}.

For varieties which are not necessarily smooth, the picture has not been so complete. Algebraic $K$-theory fails to associate long exact sequences to blowups in general, as easy examples show. %
The failure is in some sense a ``quasi-coherent'' part of algebraic $K$-theory. 
%
Indeed, quasi-coherent cohomology also fails to have long exact sequences for blowups. On the other hand, if one remembers infinitesimal information around the centre, one \emph{does} get long exact sequences for quasi-coherent cohomology. This is described in Grothendieck's theorem on formal functions, \cite[Thm.4.1.5]{EGAIII1}.

The analogous formal blowup sequences for algebraic $K$-theory have been studied for some 20 years now, \cite[\S 3-4]{Wei01}, \cite{KS02}, \cite{GH06}, \cite{GH11}, \cite{Mor18}, and feature in Kerz, Strunk, Tamme's celebrated proof of Weibel's conjecture, \cite{KST18}; see \cite{Mor16} for a historical overview. However, the absence of an associated topology has meant that the topos theoretic techniques so skillfully employed by Voevodsky were not available.

In this article we propose the following definition for a procdh topos. 

\begin{defi}[Definition~\ref{defi:procdh}]
Let $S$ be a qcqs scheme. The \emph{procdh topology} on the category $\Sch_S$ of $S$-schemes of finite presentation is generated by the following coverings.
\begin{enumerate}
 \item Distinguished Nisnevich coverings: families of the form
\[  \{ U \stackrel{i}{\to} X, V \stackrel{j}{\to} X\} \]
 such that $i$ is a quasi-compact open immersion, $j$ is an étale morphism, and $j$ is an isomorphism over $X \setminus U$ equipped with any (equivalently all) closed subscheme structure(s) of finite presentation.

 \item Proabstract blowup squares: families of the form
 \[ \{ Z_n \to X\}_{n \in \NN} \sqcup \{Y \to X\} \]
 where $Y \to X$ is a proper morphism of finite presentation which is an isomorphism outside of a closed subscheme $Z_0 \subseteq X$ of finite presentation, and $Z_n = \uSpec \OO_X / \sI_Z^n$ is the $n$th infinitesimal thickening of $Z_0$.
\end{enumerate}
\end{defi}

A sign that this is a ``correct'' topology is that it captures procdh excision: presheaves with the long exact sequences mentioned above are precisely the procdh sheaves.

\begin{theo}[Theorem~\ref{theo:descentConditions}, Corollary~\ref{coro:procdhHypercomplete}, Proposition~\ref{prop:hyperExciPoint}] \label{theo:decsentStateIntro}
Let $S$ be a scheme and consider the following conditions on a presheaf of spaces $F \in \PSh(\Sch_S, \cS$).%
\footnote{As usual, $\cS := N\cK an$ is the quasi-category associated to the simplicial category of Kan complexes, and $\PSh(\Sch_S, \cS)$ is the quasi-category $\Fun(\Sch_S^{\op}, \cS)$. Before \cite{HTT}, the category $\PSh(\Sch_S, \cS)$ would have been the category $\PSh(\Sch_S, \Set_{\Delta})$ of presheaves of simplicial sets with any model structure for which weak equivalences are objectwise weak equivalences. If one reads the statements using this latter interpretation, all limits should be replaced with homotopy limits as described in \cite{BK72} or \cite{Hir03}.}.
\begin{enumerate}
 \item
 \textnormal{
 \emph{Excision.} 
 For every distighished Nisnevich square $\{U \to X, V \to X\}$ and every proabstract blowup square $\{Z_n \to X\}_{n \in \NN} \sqcup \{Y \to X\}$ in $\Sch_S$, Def.\ref{defi:procdh}, we have
\begin{align}
F(X) &\stackrel{\sim}{\to} F(U) \times_{F(U \times_X V)} F(V), %
\\
F(X) &\stackrel{\sim}{\to} F(Y)\times_{\lim_n F(Z_n \times_X Y)} {\lim}_n F(Z_n).
\label{equa:procdhExcision}
\end{align}
}

 \item
 \textnormal{
 \emph{\v{C}ech descent.} 
 For every procdh covering family $\{Y_λ \to X\}_{λ \in \Lambda}$ we have 
 \begin{equation}
 F(X) \stackrel{\sim}{\to} \lim_{n} 
F(\cY_n)
 \end{equation}
where we write $F(\cY_n)$ for $\prod_{i \in I^n} F(Y_{i_1} \times_X \dots \times_X F_{i_n})$.
That is, in the terminology of \cite{HTT}, $F$ is a procdh sheaf.
}

 \item
 \textnormal{
 \emph{Hyperdescent.} 
 For every $X \in \Sch_S$ and procdh hypercovering $\cY_\bullet \to X$ we have 
 \begin{equation}
 F(X) \stackrel{\sim}{\to} \lim_n \Map(\cY_n, F).
\end{equation}
In the terminology of \cite{HTT}, $F$ is a hypercomplete procdh sheaf.
}
\end{enumerate}
If $S$ is qcqs then (Excision) $\Leftrightarrow$ (\v{C}ech descent). If $S$ is qcqs, has finite valuative dimension and Noetherian topological space, then (\v{C}ech descent) $\Leftrightarrow$ (Hyperdescent).
\end{theo}

Our procdh ∞-topoi have finite homotopy dimension, albeit not quite as optimal as one could hope. 

\begin{theo}[Theorem~\ref{theo:bounded}, Example~\ref{exam:counterCohDim}]
Let $S$ be a qcqs scheme of finite valuative dimension $d \geq 0$ with Noetherian underlying topological space. 
Then $\Shv_{\procdh}(\Sch_S, \cS)$ has homotopy dimension $\leq 2d$. 

There exists a Noetherian scheme of Krull dimension one with procdh homotopy dimension two.
\end{theo}

The appearance of $2d$ instead of $d$ is essentially because the topos $\PSh(\NN, \cS)$ has homotopy dimension one. So every time we perform a $\lim_{\NN}$, for example in Eq.\eqref{equa:procdhExcision} above,  we potentially increase the homotopy dimension by one. In a future article we consider a way around this using a site involving formal schemes. 

Despite the unwanted factor of two, finite homotopy dimension is sufficient to imply that the topos is hypercomplete.


\begin{coro}[Corollary~\ref{coro:procdhHypercomplete}]
Let $S$ be a qcqs scheme of finite valuative dimension with Noetherian underlying topological space. Then $\Shv_{\procdh}(\Sch_S, \cS)$ is hypercomplete.
\end{coro}

Closely related to hypercompleteness is existence of enough points. Recall that a topos, resp. ∞-topos $T$, is said to have \emph{enough points} when the collection of all geometric morphisms of topoi $\phi^*: T \to \Set$, resp. ∞-topoi $\phi^*: T \to \cS$, detect isomorphisms, resp. equivalences.

\begin{theo}[Theorem~\ref{theo:procdhEnoughPoints}, Corollary~\ref{coro:enoughInfinityPoints}] \label{theo:enoughPointsIntro}
Suppose $S$ is a qcqs scheme with Noetherian topological space of finite Krull dimension. Then both the classical topos $\Shv_{\procdh}(\Sch_S)$, and the ∞-topos $\Shv_{\procdh}(\Sch_S, \cS)$ have enough points.
\end{theo}

We prove the classical version of Theorem~\ref{theo:enoughPointsIntro} first and then upgrade it to the ∞-version. Something worth noting is that since the generating covering families of the procdh topology are not finite, the topos is not coherent, so we cannot simply apply Deligne's completeness theorem. We use the following proposition for the upgrade. For sites whose associated topos is coherent, Prop.\ref{prop:enoughPoints} plus Deligne's theorem imply Lurie's version, \cite[Thm.A.4.0.5]{SAG}.

\begin{prop}[Proposition~\ref{prop:enoughPoints}]
Let $(C, τ)$ be a small site admitting finite limits such that the topos $\Shv(C, \Set)$ has enough points in the classical sense. Then the hypercompletion $\Shv(C, \cS)^\wedge$ has enough points as an ∞-topos.
\end{prop}

After all this talk of procdh points, the reader is surely asking for a characterisation of the local rings. Procdh local rings are essentially versions of cdh local rings---henselian valuation rings---with some mild nilpotent thickening.

\begin{prop}[Characterisation of procdh local rings] \label{prop:procdhLocal}
An affine $S$-scheme $\Spec(R) \to S$ is procdh local, Def.\ref{defi:tauLocal}, if and only if it is of the form
\[ R \cong \OO \times_K A \]
with $A$ a local ring of Krull dimension zero and $\OO \subseteq K$ a henselian valuation ring of $K {=} A / \m_A$. In other words, $R \subseteq A$ is the set of elements whose residue mod $\m_A$ lies in $\OO$.
\end{prop}

We also discuss how a general procdh local ring can often be written as a filtered colimit of smaller local rings, see Proposition~\ref{prop:colimOfSmaller} and Theorem~\ref{theo:colimfgNilredhvr}.

In the latter sections, we apply the above theory to show the following topos-theoretic interpretation of the Bass construction.

\begin{theo}[Theorem~\ref{thm;apcdhKconn=K}]
For any Noetherian scheme $X$ with $\dim(X)<\infty$, there exists a natural equivalence
\[ (a_{\procdh} τ_{\geq 0}K)(X) \simeq K(X).\]
Here $K(X)$ is the non-connective algebraic $K$-theory of $X$ and $τ_{\geq 0}K(X)$ the connective $K$-theory.
\end{theo}

The above equivalence is an analogue of the equivalence%
\footnote{%
That $KH$ has cdh descent 
was proven by Haesemeyer, \cite{Hae04},
for finite type schemes over a characteristic zero base field 
and for finite dimensional Noetherian schemes by Cisinski in \cite[Thm.3.9]{Cis13} using Ayoub's proper base change theorem. Given this cdh descent, the equivalence $(a_{cdh} \tau_{\geq 0})K(X) \simeq KH(X)$ follows from resolution of singularities (if existing), connectivity of $K(X)$ for $X$ regular, \cite[Prop.6.8]{TT90}, and the equivalence $K_{\geq 0}(X) \cong KH_{\geq 0}(X)$ for $X$ regular Noetherian, \cite[Thm.2.7]{Qui73}, \cite[Exa.1.4]{Wei89}. In the absence of resolution of singularities it can be deduced from homotopy invariance and connectivity for $K$-theory of valuation rings, \cite[Thm.1.3]{KM}, and the fact that hensel valuation rings form a conservative family of fibre functors for the cdh topology, \cite{GL01}, \cite{GK15}. The equivalence $(a_{cdh} \tau_{\geq 0})K(X) \simeq KH(X)$ for finite dimensional Noetherian schemes is also proven in the landmark Kerz-Strunk-Tamme paper \cite[Thm.6.3]{KST18}.
}
\[ (a_{cdh} \tau_{\geq 0})K(X) \simeq KH(X) \]
with $X$ still finite dimensional and Noetherian.

Using Theorem~\ref{thm;apcdhKconn=K} we propose the following construction of a non-$\AA^1$-invariant motivic cohomology. For $X$ smooth over a field $\FF$, let $Fil^n_{mot} K(X)$ denote the motivic filtration on $K$-theory whose associated spectral sequence is the Atiyah-Hirzebruch spectral sequence, \cite{FrSus} and \cite{Lev08},
\[ E_2^{p,q} = H_{\cM}^{p-q}(X, \ZZ(-q)) \implies K_{-p-q}(X). \]
Here $H_{\cM}^{m}(X, \ZZ(n)) = R^m\Gamma_{\Zar}(X, \Znsm)$ with $\Znsm \in \PSh(\Sm_\FF, D(\ZZ))$ the graded pieces of $Fil^n_{mot} K$ on the category $\Sm_\FF$ of smooth $\FF$-schemes. Note that $\Znsm$ is identified with Voevodsky's $\AA^1$-invariant motivic complex and Bloch's cycle complex.

\begin{defi}[Definition~\ref{def;Znpcdh}]
For integers $n\geq 0$, we define the \emph{procdh local motivic complex}
\[\Znpcdh:=a_{\procdh} \Lsm \Znsm \in \Shv_{\procdh}(\Sch^\qcqs_\FF,D(\ZZ))),\]
as the procdh sheafification of the left Kan extension $\Lsm \Znsm$ of $\Znsm$ along $\Sm_\FF \to \Sch^{\qcqs}_\FF$. Here $\Sch^{\qcqs}_\FF$ is the category of qcqs $\FF$-schemes.
\end{defi}

It follows essentially from the definition, together with a result of Bhatt-Lurie and the fact that the procdh site of a Noetherian scheme has finite cohomological dimension that we get an Atiyah-Hirzebruch spectral sequence.

\begin{theo}[Theorem~\ref{thm;AH-Kpcdh}]
For any Noetherian $\FF$-scheme $X$ with $\dim(X)<\infty$, there exists a 
complete 
multiplicative 
 decreasing $\NN$-indexed filtration $\biggl\{\Fpcdh^n K(X)\biggl\}_{n\in \NN}$ on $K(X)$ and identifications 
\[ \gr_{\Fpcdh}^n K(X) \simeq \Znpcdh(X)[2n].\]
\end{theo}

Recently, in \cite{EM} Elmanto-Morrow have also proposed a non-$\AA^1$-invariant motivic cohomology for qcqs $\FF$-schemes. We will write $\Znem$ for these presheaves. They are constructed by modifying the cdh sheafification $\Zncdh$ of the left Kan extension of $\Znsm$ along $\Sm_\FF \to \Sch^\qcqs_\FF$ by using Hodge-completed derived de Rham complexes in case $\FF=\QQ$ and syntomic complexes in case $\FF=\FF_p$.
The construction is motivated by trace methods in algebraic $K$-theory using the cyclotomic trace map $\tr$.

We conclude the article with the following comparison.

\begin{coro}[Corollary~\ref{cor;Znem-comparison}]
Let $\FF = \FF_p$ or $\QQ$ and take a Noetherian $\FF$-scheme $X$. Then there are equivalences 
\[ \Znpcdh(X) \simeq \Znem(X), \]
functorial in $X$.
 \end{coro}

\subsection{Future work} As mentioned above, a version of the material in this article for formal schemes is in progress. We also have a version for derived schemes.

\subsection{Notation and conventions}

Throughout we write $\Sch_S$ for the category of $S$-schemes of finite presentation over a scheme $S$.  %
We write $X^{\gen}$ for the set of generic points of a scheme $X$ equipped with the topology induced from the underlying topological space of $X$.

The first sections deal exclusively with presheaves of sets. When we pass to presheaves of spaces we will write $\PSh(-, \cS), \PSh(-, \Spt), \PSh(-, D(\ZZ))$, etc, and sometimes use $\PSh(-, \Set)$ if we want to emphasise that a presheaf takes values in discrete spaces. 

\subsection{Acknowledgements}

We are very grateful to Matthew Morrow for many interesting questions and discussions about the procdh topology and particularly for a very active scientific exchange without which the comparison theorem would not have been possible.

We thank Marc Hoyois for answering the first author's annoying emails about infinity categories.

We also thank Dustin Clausen and Ryomei Iwasa for their encouraging interest in the project.

Much of the research in this article was conducted while the first author was supported by JSPS KAKENHI Grant (19K14498) and Bilateral Joint Research Projects (120213206). The second author is supported by JSPS Grant-in-aid (B) \#20H01791 representative Shuji Saito.

\section{Definition of a procdh topology}

In this section we define a procdh topology.

\begin{defi} \label{defi:procdh}
Let $S$ be a scheme. The \emph{procdh topology} on $\Sch_S$ is generated by the following coverings.
\begin{enumerate} \setcounter{enumi}{-1}
 \item Zariski coverings.

 \item Distinguished Nisnevich coverings: families of the form
\[  \{ U \stackrel{i}{\to} X, V \stackrel{j}{\to} X\} \]
 such that $i$ is a quasi-compact open immersion, $j$ is an étale morphism, and $j$ is an isomorphism over $X \setminus U$ equipped with any (equivalently all) closed subscheme structure(s) of finite presentation.

 \item Proabstract blowup squares: families of the form
 \[ \{ Z_n \to X\}_{n \in \NN} \sqcup \{Y \to X\} \]
 where $Y \to X$ is a proper morphism (of finite presentation) which is an isomorphism outside of a closed subscheme $Z_0 \subseteq X$ of finite presentation, and $Z_n = \uSpec \OO_X / \sI_Z^n$ is the $n$th infinitesimal thickening of $Z_0$.
\end{enumerate}
\end{defi}

By convention the empty covering of the empty scheme is also a covering family. %
The following remarks can be skipped on a first reading.

\begin{rema} \label{rema:genDefi}
More precisely, a family $\cY = \{Y_i \to X\}_{i \in I}$ in $\Sch_S$ is a procdh covering if it can be refined by a composition of finite length of families of the above form. 

Even more precisely, $\cY$ is a procdh covering if there exists a rooted tree $T$ of finite height and a functor $T \to \Sch_S$, such that the root is sent to $X$, for each vertex $v$ the family $\{ W_c \to W_v \}_{c \in Child(v)}$ is of one of the above forms, and for each leaf $l$ there exists a factorisation $W_l \to Y_{i_{l}} \to X$ for some $i_l \in I$. 
\end{rema}

\begin{rema}
If $S$ is quasicompact, then all $S$-schemes of finite presentation are also quasicompact, so the $\procdh$ topology is generated by just distinguished Nisnevich coverings and proabstract blowup squares.
\end{rema}

\begin{rema}
Since the proabstract blowup square family $\{ Z_n \to X\}_{n \in \NN} \sqcup \{Y \to X\}$ contains the subfamily $\{ Z_0 \to X, Y \to X\}$, every cdh sheaf is a procdh sheaf.
\end{rema}



\begin{rema}
If one replaces $\Sch_S$ with the category $\Sch^{\qcqs}$ of quasi-compact quasi-separated schemes then one gets the coarsest topology such that for all $X \in \Sch^{\qcqs}$ the functor $\Sch_X \to \Sch^{\qcqs}$ is a continuous morphism of sites, \cite[00WV]{stacks-project}, \cite[Def.III.1.1]{SGA41}. Fibre functors of the topos $\Shv_{\procdh}(\Sch^{\qcqs})$ are inconvenient as they correspond to ind-rings, not only rings. Indeed, there are easy examples of non-zero sheaves of abelian groups in $\Shv_{\procdh}(\Sch^{\qcqs})$ which are zero when evaluated on every procdh local ring---a phenomenon also occurring in the Zariski topology.
Of course, for presheaves $F$ that preserve filtered colimits of rings, such as $K$-theory, the fibre $\colim_λ F(R_λ)$ at an ind-ring $``\colim" R_λ$ is isomorphic to the fibre $F(\colim R_λ)$ at its colimit $\colim R_λ$.
\end{rema}

\begin{warn}
In the cdh topology, the composition $L'$ of left Kan extension $L$ with sheafification fits into a commutative diagram
\[ \xymatrix{
\PSh(\Sch_\ZZ) \ar[r]^-L & \PSh(\Sch^{\qcqs}) \\
\Shv_{\cdh}(\Sch_\ZZ) \ar[r]^-{L'} \ar[u] & \Shv_{\cdh}(\Sch^{\qcqs}) \ar[u] 
} \]
because, up to refinement, any cdh-covering is the pullback of one in $\Sch_\ZZ$, and due to the fact that sheaves are characterised by excision squares. The corresponding procdh square does not necessarily commute.
\end{warn}

\begin{warn}
Similarly, when $k$ admits resolutions of singularities, the square
\[ \xymatrix{
\PSh(\Sm_k) \ar[r]^-L & \PSh(\Sch_k) \\
\Shv_{\cdh}(\Sm_k) \ar[r]^-{L'} \ar[u] & \Shv_{\cdh}(\Sch_k) \ar[u] 
} \]
is commutative, where $\Sm_k$ is equipped with the induced topology and, again, $L'$ is the composition of left Kan extension with sheafification. We see no reason for the corresponding procdh square to commute.
\end{warn}


\section{Local rings}

In this section we characterise procdh local rings, Prop.\ref{prop:procdhLocal}. We also discuss how a general procdh local ring can often be written as a filtered colimit of smaller local rings, Prop.\ref{prop:colimOfSmaller}, Theo.\ref{theo:colimfgNilredhvr}.
The latter is used in the proof of the comparison theorem, Corollary~\ref{cor;Znem-comparison}.

\begin{defi} \label{defi:FF}
Recall that a \emph{fibre functor} of a topos $\Shv_τ(C)$ is a continuous morphism of topoi $\phi^*: \Shv_τ(C) \rightleftarrows \Set: \phi_*$, or equivalently, a functor $\phi^*: \Shv_τ(C) \to \Set$ which preserves colimits and finite limits. 
\end{defi}

For topologies $τ$ on $\Sch_S$ for which every scheme is covered by affine ones (e.g., anything finer than the Zariski topology), there is a bijection between fibre functors of $\Shv_\tau(\Sch_S)$ and affine $S$-schemes $\Spec(R) \to S$ which are \emph{$τ$-local} in the following sense.

\begin{defi}[{cf.\cite[Thm.III.4.1]{SGA41}, 
\cite[I.8.10.14]{SGA41}, 
\cite[Cor.8.13.2]{EGAIV3}
}] \label{defi:tauLocal}
Let $τ$ be a topology on $\Sch_S$ such that every scheme is covered by affine ones. An affine $S$-scheme $\Spec(R) \to S$ is said to be \emph{$τ$-local} if for every $τ$-covering $\{Y_i \to X\}_{i \in I}$ the morphism of sets
\begin{equation} \label{equa:RYXEpi}
\coprod_{i \in I} \hom(\Spec(R), Y_i) \to \hom(\Spec(R), X)
\end{equation}
is surjective.
\end{defi}

\begin{exam}
The cdh local $S$-schemes are those $\Spec(R) \to S$ such that $R$ is a henselian valuation ring, \cite{GK15}, \cite{GL01}. 
\end{exam}

\begin{rema}
For many topologies, including the procdh topology as we will see below in Thm.\ref{theo:procdhEnoughPoints}, covering families are detected by the collection of fibre functors, \cite[Prop.IV.6.5]{SGA41}. That is, we could have defined procdh coverings as those families such that for every procdh local ring, Eq.\eqref{equa:RYXEpi} is surjective.
\end{rema}

\begin{prop}[Characterisation of procdh local rings] \label{prop:procdhLocal}
An affine $S$-scheme $\Spec(R) \to S$ is procdh local, Def.\ref{defi:tauLocal}, if and only if it is of the form
\[ R \cong \OO \times_K A \]
with $A$ a local ring of Krull dimension zero and $\OO \subseteq K$ a henselian valuation ring of $K {=} A / \m_A$. In other words, $R \subseteq A$ is the set of elements whose residue mod $\m_A$ lies in $\OO$.
\end{prop}

\begin{rema}
Explicitly, to get the fibre functor associated to a ring $R$ one first replaces $R$ with the pro-object $``\colim"_{\Spec(R) \to X \to S} X$ of $\Sch_S$ where the colimit is over factorisations with $X \in \Sch_S$. Then the fibre functor is evaluation of a sheaf $F \in \Shv_{\procdh}(\Sch_S)$ on this pro-object $F \mapsto ``\colim"_{\Spec(R) \to X \to S} F(X)$. 

This distinction between pro-objects and their limits becomes important if one wants to work with, say, the procdh site of all qcqs schemes with Noetherian topological space, since in this case, evaluation on a pro-object and evaluation on its limit give different answers in general.
\end{rema}

\begin{rema}
Since Nisnevich coverings are procdh coverings, every procdh local ring is necessarily a Nisnevich local ring. That is, rings of the form described in Prop.\ref{prop:procdhLocal} are henselian.
\end{rema}

\begin{rema}
For a ring of the above form, the ideal $\n = \Nil(R)$ of nilpotents is the unique minimal prime, and we necessarily have 
\[ A = R_\n,\  K = k(\n),\  \OO = R / \n. \]
Writing $Q(R)$ for the ring of total quotients, we also have 
\[ A = Q(R),\  K = Q(R)_{\red},\  \OO = R_\red. \]
\end{rema}

\begin{exam}\ 
\begin{enumerate}
 \item Every henselian valuation ring is a procdh local ring; in this case $R_\n = k(\n)$.
 \item Every dimension zero local ring is a procdh local ring; in this case
 $R/\n = k(\n)$.
 \item The ring $\ZZ_p + \QQ_p \epsilon \subseteq \QQ_p[\epsilon] / \epsilon^2$ is a procdh local ring.
\end{enumerate}
\end{exam}


\begin{proof}[Proof of Proposition~\ref{prop:procdhLocal}]
%
($\Rightarrow$) 
The following composition of a Zariski covering and the procdh covering 
\begin{equation}
\left \{ \uSpec \frac{\OO_S[x,y]}{\langle x^n, y^n  \rangle} \right \}_{n \in \NN} 
\sqcup 
\biggl \{ \uSpec \OO_S[x,\tfrac{y}{x}], \uSpec \OO_S[\tfrac{x}{y},y] \biggr \}
\end{equation}
shows that procdh local rings satisfy: 
\begin{enumerate}
 \item[{($*$)}] $\forall a, b \in R$;  we have $a | b$ or $b|a$ or $a$ and $b$ are both nilpotent.
\end{enumerate}
It follows from this that $R_{\red}$ is a valuation ring, and in particular, $R$ has a unique minimal prime ideal $\n$, which equals the set of nilpotents. All zero divisors are nilpotent by virtue of the procdh covering
\begin{equation}
\left \{ \uSpec \frac{\OO_S[x,y]}{\langle x^n, xy  \rangle} \right \}_{n \in \NN} \sqcup 
\left \{ \uSpec \frac{\OO_S[x,y]}{\langle y \rangle} \right \}
\end{equation}
of $\uSpec \frac{\OO_S[x,y]}{\langle xy \rangle}$, so $R \to R_\n$, and therefore $R \to (R / \n) \times_{k(\n)} R_\n$ is injective. We claim that the latter is also surjective. It follows from a diagram chase that $\n \to \n R_\n$ is surjective implies $R \to (R / \n) \times_{k(\n)} R_\n$ is surjective. 
\begin{equation} \label{equa:nRnRnR}
\xymatrix@R=12pt{
0 \ar[r] &
\n \ar[r] \ar[d] & 
R \ar[r] \ar[d] &
R / \n \ar[r] \ar[d] & 
0 \\
0 \ar[r] & 
\n R_\n \ar[r] &
R_\n \ar[r] &
R_\n / \n R_\n \ar[r] &
0
}
\end{equation}
So suppose we have $a \in \n$ and $s \in R \setminus \n$. We claim there is $b \in \n$ such that $b/1 = a/s$. Indeed, this follows from $a \in \n$, $s \in R \setminus \n$,  and ($*$).

So we have shown $R \to (R / \n) \times_{k(\n)} R_\n$ is both injective and surjective. The Krull dimension of $R_\n$ is zero because $\n$ is a minimal prime, and we have already observed that $R/\n = R_\red$ is a valuation ring, so it suffices to show that $R/\n$ is henselian. But procdh local rings are Nisnevich local rings, also known as henselian local rings, and quotients of henselian local rings are henselian local rings.

($\Leftarrow$). Suppose $R = \OO \times_K A$ as in the statement. 
We want to show that \eqref{equa:RYXEpi} is an epimorphism for all procdh coverings. Certainly it suffices to consider the generator coverings described in the definition, cf. Remark~\ref{rema:genDefi}.

%
%

We immediately notice that $R$ is henselian: $\OO = R_{\red}$ is henselian by assumption and $-\otimes_R (R_{\red})$ induces an equivalence ${\operatorname {Et}}_R \stackrel{\sim}{\to} {\operatorname {Et}}_{R_{\red}}$ of categories of étale algebras, \cite[039R]{stacks-project}.  So the desired lifting condition with respect to Nisnevich coverings is satisfied, cf.\cite[04GG, Item(7)]{stacks-project}.

Suppose we have a proabstract blowup square $\{Z_n \to X\}_{n \in \NN} \sqcup \{ Y \to X\}$ and a morphism $\Spec(R) \to X$. If the generic point $\Spec(K)$ of $\Spec(R)$ doesn't land in $Z_0$, then it lifts through $Y$ because $Y \to X$ is an isomorphism over $X \setminus Z_0$. By the valuative criterion for properness, this lifting extends to a lifting of $\Spec(\OO)$. Since $A$ is local %
the morphism $\Spec(K) \to Y$ also extends to $\Spec(A) \to Y$. These three morphisms factor through some open affine of $Y$, so they glue to give a lifting $\Spec(R) \to Y \to X$ since $\Spec(\OO \times_K A) = \Spec(\OO) \sqcup_{\Spec(K)} \Spec(A)$ is the categorical pushout in the category of affine schemes.

On the other hand, if $\Spec(K) \to X$ does factor through $Z_0$, then $\Spec(\OO) \to X$ also factors through $Z_0$. The morphism $\Spec(A) \to X$ doesn't necessarily factor through $Z_0$ but the (finitely many) generators of $\sI_{Z_0}$ are sent inside $\m_A = $ Nilpotents$(A)$ so $\Spec(A) \to X$ factors through some $Z_n$. Then we glue as in the previous case.
\end{proof}

\begin{rema}[{Cf.\cite{Kel23}}] \label{rema:genprocdh}
The proof of Proposition~\ref{prop:procdhLocal} shows that an $S$-scheme $\Spec(R) \to S$ is procdh local if and only if it is local with respect to Nisnevich coverings $\{V_i \to S\}_{i \in I}$ and the two coverings
\begin{equation}
\tag{2a} \label{equa:twodash}
\{\{0\}_n \to \AA^2\}_{n \in \NN} \sqcup \{Bl_{\AA^2} \{0\} \to \AA^2\}
\end{equation} 
and 
\begin{equation}
\tag{2b} \label{equa:twodash}
\left \{ \uSpec \frac{\OO_S[x,y]}{\langle x^n, xy  \rangle} \right \}_{n \in \NN} \sqcup 
\left \{ \uSpec \frac{\OO_S[x,y]}{\langle y \rangle} \right \}.
\end{equation}
If the topology generated by these coverings has enough points, e.g., if $S$ is countable, \cite[Theorem 6.2.4]{MR06}, then it follows that these coverings generate the procdh topology. For more details see \cite{Kel23}, specifically Observation 2.13.
\end{rema}

General procdh local rings can be quite large, but they are often filtered colimits of smaller local rings. The rest of this section is devoted to such reductions and some consequences, for example Corollary~\ref{coro:smallFamily}.

\begin{rede}[Prorh local rings]
One can define a version of the procdh topology omitting the Nisnevich coverings from Definition~\ref{defi:procdh} (but keeping Zariski ones). This could be called the \emph{prorh topology}. The same argument as in Proposition~\ref{prop:procdhLocal} then characterises prorh local rings as those rings of the form $\OO \times_K A$ with $A$ a local ring of dimension zero, and $\OO \subseteq K$ a (not necessarily henselian) valuation ring of $K = A /\m$. For the purposes of this article, let's just define a prorh local ring to be a ring of this form.
\end{rede}

One way to construct procdh local rings is to first build a prorh local ring, and then take the henselisation.

We write $Q(R)$ for the ring of total fractions. That is, $Q(R) = R[S^{-1}]$ where $S$ is the set of nonzero divisors.

\begin{prop} \label{prop:etaleProrh}
Suppose that $R$ is a prorh local ring. 
That is, $R = \OO \times_K A$ with $\OO$ a (not necessarily henselian) valuation ring, $A$ a dimension zero local ring, and $K = \Frac(\OO) \cong A/\m$. Let $R \to S$ be an étale morphism towards a local ring $S$.

Then $S$ is prorh local, and the henselisation $R^h$ is a procdh local ring.
Moreover, in this situation we have $\dim S \leq \dim R$, $\length Q(R) = \length Q(S)$ and $\dim R^h = \dim R$, $\length Q(R) = \length Q(R^h)$.
\end{prop}

\begin{proof}
The canonical morphism $S \to (S \otimes_R \OO) \times_{(S \otimes_R K)} (S \otimes_R A)$ is an isomorphism because $R \to S$ is flat, \cite[Thm.2.2(iv)]{Fer03}. Since $S$ is local, $S \otimes_R \OO = S \otimes_R (R_\red)$ is also local, and therefore a valuation ring of Krull dimension $\leq \dim \OO$, \cite[0ASJ]{stacks-project}. It follows that the étale $K$-algebra $S \otimes_R K$ is actually a finite separable extension of $K$, and since $(S \otimes_R A)_{\red} = (S \otimes_R (A_{\red}))_{\red} = (S \otimes_R K)_{\red}$, we find that $S \otimes_R A$ has exactly one prime ideal. %
Since $K \to S \otimes_R K$ is a field extension and $R \to S$ is flat, for any composition series $A \supset I_0 \supset I_1 \supset \dots$, the pullback $S \otimes_R A  \supset S \otimes_R I_0 \supset S \otimes_R I_1 \supset \dots$ is a composition series for $S \otimes_R A$. So $\length A = \length S \otimes_R A$.

In the case that $R \to S$ is a local homomorphism, $R_\red \to S_\red$ is also local, so it induces a bijection of value groups of these valuation rings, \cite[0ASF]{stacks-project}, and therefore a bijection of posets of prime ideals. Consequently, $\dim R = \dim R_\red = \dim S_\red = \dim S$. Now the henselisation $R^h$ is the colimit $\colim S_λ$ over étale algebras $R \to S_λ$ which are local homomorphisms of local rings. We have just seen that all transition morphisms $S_λ \to S_{μ}$ induce homeomorphisms on $\Spec$, so we deduce that the prime ideals of $R^h$ are in bijection with the prime ideals of $R$. Therefore, $\dim R^h = \dim R$. %
As above, since $K \to R^h \otimes_R K$ is a field extension and $R \to R^h$ is flat, pullback preserves composition series so $\length A = \length R^h \otimes_R A$.

It remains to show that $R^h$ is procdh local. We have $(R^h)_\red = (\colim S_λ)_\red = \colim S_{λ, \red}$ is a filtered colimit of valuation rings and therefore a valuation ring. It follows that $\n^h := \Nil(R^h)$ is a unique minimal prime ideal of $R^h$, and also that $R^h / \n^h$ is a \emph{henselian} valuation ring, since quotients of henselian rings are henselian. To conclude it suffices to show that $R^h \to (R^h / \n^h) \times_{k(\n^h)} (R^h_{\n^h})$ is an isomorphism. This follows the fact that filtered colimits commute with fibre products.
\end{proof}

\begin{prop} \label{prop:colimOfSmaller}
If $S$ is a Noetherian scheme (e.g., the spectrum of $\ZZ$ or a field) and $\Spec(R) \to S$ an affine $S$-scheme such that $R$ is procdh local, then $R$ is a filtered colimit of procdh local $\OO_S$-algebras such that each $\length Q(R_λ)$ is finite. If $S$ has finite valuative dimension, Def.\ref{defi:valuDIm}, then we can also assume each $R_λ$ has finite Krull dimension. 
\end{prop}

\begin{proof}
We immediately observe that since $R$ is local, we can assume $S$ is affine, say $S = \Spec(B_0)$. Then the idea is that for any factorisation $B_0 \to B \to R$, we can convert $B$ into a procdh local ring in a way which is functorial in $B$. If we write $R$ as a filtered colimit $R = \colim B_λ$ of $B_0$-algebras $B_λ$ of finite presentation, we can use this fact to functorially convert each $B_λ$ into a procdh local ring $R_λ$, giving the expression $R = \colim R_λ$. To conclude we observe that the $R_λ$ have the properties described in the statement if $B_0$ is Noetherian, resp. of finite valuative dimension. 

Now we carry out this plan. Let $B_0 \to B_λ \to R$ be any factorisation, and consider the preimage $\p_λ \subseteq B_λ$ of the minimal prime $\n$ of $R$. The valuation ring $R/\n$ of $R$ induces and factors through a valuation ring $\OO_λ = k(\p_λ) \cap (R / \n)$ of $k(\p_λ)$, and the localisation $R_\n$ of $R$ factors through the localisation $(B_λ)_{\p_λ}$. Form the epi-monic factorisation $(B_λ)_{\p_λ} \to\!\!\!\!\to A_λ \hookrightarrow R_\n$. 
As $\n R_\n = \Nil(R_\n)$, the preimage of this in $A_λ$ also consists entirely of nilpotents. Since this ideal of nilpotents is the maximal ideal of $A_λ$, it follows that $A_λ$ is a dimension zero local ring.  Glueing, we get a prorh local ring $P_λ := \OO_λ \times_{k({\p_λ})} A_λ$ and a commutative diagram
\[ \xymatrix@!R=6pt{
B_λ \ar[dr] \ar[rr] \ar[dd] && \overset{P_λ = }{\OO_λ \times_{k({\p_λ})} A_λ} \ar[dr] \ar[rr] \ar[dd]|\hole && R \ar[dr] \ar[dd]|\hole  && \\
& B_λ/{\p_λ} \ar[rr] \ar[dd] && \OO_λ \ar[dd] \ar[rr] && R / \n \ar[dd] & \\
(B_λ)_{\p_λ} \ar[dr] \ar[rr]|(0.45)\hole && A_λ \ar[dr] \ar[rr]|(0.6)\hole && R_\n \ar[dr] & \\
& k({\p_λ}) \ar@{=}[rr] && k({\p_λ}) \ar[rr] && k(\n) 
} \]
To conclude we take $R_λ = P_λ^h$ to be the henselisation, which is a procdh local ring by Proposition~\ref{prop:etaleProrh}. 

One checks that the choice of $\p_λ = $preimage$(\n) \subseteq B_λ$, and the constructions of $k({\p_λ})$, $\OO_λ = k(\p_λ) \cap (R / \n)$, $(B_λ)_{\p_λ}$, $A_λ = $image$\left ((B_λ)_{\p_λ} \to R_\n \right)$, $P_λ = \OO_λ \times_{k(\p_λ)} A_λ$, and $R_λ = P_λ^h$ are functorial in $B_λ$. Consider the system of factorisations $B_0 \to B_λ \to R$ with $B_0 \to B_λ$ of finite presentation. Note, $\colim B_λ = R$.
We get a factorisation 
\[ \colim B_λ \to \colim P_λ \to \colim R_λ \to R. \]
The composition is an isomorphism by construction. The maps $A_λ \to R_\n$ and $\OO_λ \to R / \n$ are injective by construction, so each $P_λ  = \OO_λ \times_{k(\p)} A_λ \to R$ is injective, so $ \colim P_λ \to R$ is injective, and therefore also an isomorphism. Finally, henselisation commutes with filtered colimits of local rings with local transition homomorphisms, \cite[07RP]{stacks-project}. %
So $\colim R_λ = \colim P_λ^h = (\colim P_λ)^h = R^h = R$.

If $B_0$ is Noetherian, then each $B_λ$ is also Noetherian, $(B_λ)_{\p_\lambda}$ is Noetherian, and the quotient $A_λ = Q(P_λ)$ is Noetherian, hence, of finite length. Included in Proposition~\ref{prop:etaleProrh} is that $\length Q(P_λ) = \length Q(R_λ)$.

If $B_0$ has finite valuative dimension, then so does each $B_λ$ and $B_λ /\p_\lambda$ so $\dim \OO_λ = \dim P_λ$ is finite. By Proposition~\ref{prop:etaleProrh} $\dim P_λ = \dim R_λ$.
\end{proof}

Sometimes finite length $Q(R)$ is not strong enough, and we would like $\Nil(R)$ to be finitely generated. This is not achievable for general procdh local rings.

\begin{exam}
Let $R$ be a procdh local ring such that $\n = \Nil(R)$ is finitely generated. Then $R = R_\n$. That is, $R$ is an local ring of Krull dimension zero. 

As $\n$ is finitely generated, there is some $n$ such that $\n^{n+1} = 0$ and $\n^n \neq 0$.\footnote{We allow $n = 0$, in which case $\n^0 = R$.}
Since $R = (R / \n) \times_{k(\n)} (R_\n)$ is defined by a Milnor square, 
the map $R \to R_\n$ induces an identification of $\n$ with $\n R_\n$, and consequently, 
an identification of $\n^n$ with $(\n R_\n)^n$. Since $\n^{n+1}$ (and therefore $(\n R_\n)^{n+1}$) is zero, $\n^n$ (resp. $(\n R_\n)^n$) is an $R/\n$-module (resp. $k(\n)$-vector space). Moreover, since $\n$ is a finitely generated $R$-module, so is $\n^n$. 
Changing notation to $\OO = R/\n$ and $M = \n^n$ to make the consequences of this clearer, we have a finitely generated nonzero module $M$ over a valuation ring $\OO$ such that $M \to M \otimes_{\OO} \Frac(\OO)$ is an isomorphism. This is only possible if $\OO = \Frac(\OO)$.%
\footnote{For valuation rings we have torsion free $\Leftrightarrow$ flat, and %
for integral domains we have finitely generated flat $\Leftrightarrow$ finitely generated projective, and %
for local rings we have projective $\Leftrightarrow$ free, so the isomorphism in question is isomorphic to the canonical inclusion $\OO^{\oplus r} \subseteq \Frac(\OO)^{\oplus r}$ for some $r > 0$.} %
Returning to the previous notation, we have shown that $R/\n = k(\n)$, and therefore $R = (R / \n) \times_{k(\n)} (R_\n) = R_\n$. 
\end{exam}

Now we start working towards Theorem~\ref{theo:colimfgNilredhvr} which says that we can get to finitely generated nilradical if we relax the conditions on the nilpotents slighty. 
Theorem~\ref{theo:colimfgNilredhvr}  is used in the proof of the comparison theorem, Corollary~\ref{cor;Znem-comparison}.

\begin{lemm} \label{lemm:section}
Suppose $R$ is a procdh local ring containing a field $\FF$ such that $Q(R)_\red / \FF$ is separably generated. Then the canonical projection $R \to R_\red$ admits a section in the category of $\FF$-algebras.
\end{lemm}

This essentially follows from $Q(R)_\red / \FF$ being formally étale, and $Q(R) \to Q(R)_\red$ being a nilpotent thickening, but we couldn't find a reference with exactly the statement we wanted.

\begin{proof}
Since $R = \OO \times_K A$ (with $\OO = R_\red$, $K = Q(R)_\red$, $A = Q(R)$), to get a factorisation $\OO \dashrightarrow R \to \OO$ of the identity (in the category of $\FF$-algebras), it is equivalent to find a factorisation $\OO \dashrightarrow A \to K$ of the canonical inclusion $\OO \subseteq K$ (in the category of $\FF$-algebras). Furthermore, since $K = \OO[(\OO \setminus \{0\})^{-1}]$ and an element of $A$ is a unit if and only if its image in $K$ is nonzero, giving such a factorisation is equivalent to giving a factorisation $K \dashrightarrow A \to K$ of the identity (in the category of $\FF$-algebras). That is, it suffices to solve the lifting problem
\[ \xymatrix{
K & \ar[l] \ar@{-->}[dl] \FF_1 \\
A \ar[u] & \ar[l] \ar[u] \FF_0
} \]
with $\FF_0 = \FF$ and $\FF_1 = K$.
We have assumed $K / \FF$ is separably generated, so by assumption there are elements $\{x_λ \in K\}_{λ \in \Lambda}$ indexed by a well-ordered set $\Lambda$ such that for every successor $λ < λ + 1$ in $\Lambda$ the element $x_{λ+1}$ is transcendental or finite separable over the subfield $K(x_μ : μ \leq λ)$. So by transfinite induction, we can assume that the field extension $\FF_1 / \FF_0$ is generated by a single element $x \in \FF_1$ which is transcendental or finite separable.

Suppose that $\FF_1 = \FF_0(x)$ with $x$ transcendental. Choosing any element of $A$ in the preimage of $x$, we get an induced factorisation $\FF_0[x] \dashrightarrow A \to K$. Since all non-zero polynomials $f(x) \in \FF_0[x]$ are non-zero in $\FF_0(x) \subseteq K = A / \m$, they are sent to units in $A$, so we get a factorisation $\FF_0(x) \dashrightarrow A \to K$.

Now suppose $x$ is finite separable with minimal polynomial $f(X) \in \FF_0[X]$, so $\FF_1 \cong \FF_0[X] / \langle f(X) \rangle$. Choose a lift $a \in A$ of $x \in \FF_1 \subseteq K$ and note that  $f(a) \in \m$ since $f(x) = 0$. Since $\m = \Nil(A)$, we have $f(a)^n = 0$ for some $n$, so we get an induced commutative diagram
\[ \xymatrix@!R=12pt{
K  & \ar[l] \FF_0[X] / \langle f(X) \rangle &\ar@{=}[l] \FF_1 \ar@{-->}[dl] \\
A \ar[u] & \ar[l] \FF_0[X] / \langle f(X)^n \rangle \ar[u]_\phi & \ar[l] \ar[u]_\iota \FF_0
} \]
with unique diagonal since $\iota$ is étale, and $\ker(\phi)$ is a nilpotent ideal, \cite[02HM]{stacks-project}.
\end{proof}

\begin{lemm} \label{lemm:fgNil}
Suppose $\FF$ is any ring and $R$ is any $\FF$-algebra such that $R \to R_\red$ admits a section in the category of $\FF$-algebras. Then we can write $R$ as a filtered colimit $R = \colim R_λ$ of sub-$\FF$-algebras $R_λ \subseteq R$ such that $(R_λ)_{\red} \cong R_\red$ but $\Nil(R_λ)$ is finitely generated for all $λ$.
\end{lemm}

\begin{proof}
Given a set of elements $S \subseteq \Nil(R)$, write $R_\red[S] \subseteq R$ for the sub-$R_\red$-module generated by the monomials $x_{i_1}^{m_1}x_{i_2}^{m_2} \dots x_{i_n}^{m_n}$ for $x_i \in S$. Since elements of $\Nil(R)$ are nilpotent, if $S$ is finite, $\Nil(R_\red[S])$ is a finite $R_\red$-module.
As $S_\lambda$ ranges over all finite subsets of $\Nil(R)$, we obtain $R = \colim R_λ$ with $R_\lambda=R_\red[S_\lambda]$ as desired.
\end{proof}

\begin{theo} \label{theo:colimfgNilredhvr}
Every procdh local ring over a perfect field $\FF$ is a filtered colimit of $\FF$-algebras $R_λ$ (not necessarily procdh local rings) such that $\Nil(R_λ)$ is finitely generated and $(R_λ)_{\red}$ is a finite rank henselian valuation ring.
\end{theo}

\begin{proof}
We can assume our initial ring $R$ has finite Krull dimension by Prop.\ref{prop:colimOfSmaller}. The map $R \to R_{\red}$ admits a section (as $\FF$-algebras) by Lem.\ref{lemm:section}, so then the result follows from Lem.\ref{lemm:fgNil}.
\end{proof}

%
%
%
%
%
%
%

\section{Nisnevich-Riemann-Zariski spaces} \label{sec:NRZspaces}

In this section we consider a Nisnevich version 
$\RZ(X_{\Nis})$ 
of the Riemann-Zariski space 
associated to a scheme $X$.
 These will be used as small sites. The main result of this section is Corollary~\ref{coro:RZColimitLim} which says that for each $X \in \Sch_S$, the canonical comparison functor $\Shv_{\procdh}(\Sch_S) \to \Shv(\RZ(X_{\Nis}))$ preserves colimits and finite limits, at least if $S$ is a qcqs scheme with Noetherian topological space. This result will be used in Section~\ref{sec:enoughPoints} to show that the topoi have enough points, and in Section~\ref{sec:homotopyDimension} to show that the homotopy dimension is finite.

\begin{defi}
By \emph{modification} we will mean a morphism of schemes $Y \to X$ which is proper, of finite presentation, and an isomorphism over a dense qc open $D \subseteq X$. We write $\Mod_X \subseteq \Sch_X$ for the full subcategory of modifications.
\end{defi}

\begin{rema} \label{rema:ModLimits} 
 If $X$ is qcqs and $Y' \to Y$, $Y'' \to Y$ are morphisms in $\Mod_X$ then $Y' \times_Y Y''$ is again in $\Mod_X$. In particular, $\Mod_X$ admits finite limits, calculated in $\Sch_X$, and is therefore is filtered.
\end{rema}

We do not ask modifications to be birational so that finite limits in $\Mod_X$ are more nicely behaved. We can of course often refine any object in $\Mod_X$ by one which is birational to $X$, Lem.\ref{lemm:rema:genXY}\eqref{lemm:rema:genXY:bir}.

%

\begin{defi}
Let $S$ be a qcqs scheme. For $X \in \Sch_S$ we define
\[ \RZ(X_{\Nis}) = \int_{Y \in \Mod_X} Y_{\Nis}. \]
Explicitly, $\RZ(X_{\Nis}) \subseteq \Arr(\Sch_X)$ is the category whose objects are morphisms $U \to Y$ such that $U \in Y_{\Nis}$ and $Y \in \Mod_X$, and morphisms are commutative squares
\[ \xymatrix{
U' \ar[d] \ar[r] & U \ar[d] \\
Y' \ar[r] & Y
} \]
We abbreviate $U \to Y$ to $(U/Y)$.
\end{defi}

\begin{rema}
As it is a category of arrows in a category admitting finite limits, $\Arr(\Sch_X)$ admits finite limits and they are calculated component wise: $\lim(A_i/B_i) = (\lim A_i/ \lim B_i)$. If each $(A_i/B_i)$ is in $\RZ(X_{\Nis})$, then one checks that $\lim(A_i/B_i)$ is again in $\RZ(X_{\Nis})$.\footnote{We observed $\lim B_i$ is in $\Mod_X$ in Remark~\ref{rema:ModLimits}. The fastest way to check that $\lim A_i \to \lim B_i$ is étale, is probably to observe that it is formally étale and of finite presentation, \cite[02HG, 00UR]{stacks-project}.} That is $\RZ(X_{\Nis})$ admits finite limits, and they are calculated termwise.
\end{rema}

%
%

\begin{defi} \label{defi:RZtop}
The category $\RZ(X_{\Nis})$ is canonically equipped with the Grothendieck topology generated by:
\begin{enumerate}
 \item families of the form
\begin{equation} \label{equa:RZNisCov} \tag{Nis}
\{ (U_i/Y) \to (U/Y) \}_{i \in I}	
\end{equation}
such that $\{U_i \to U\}$ is a Nisnevich covering, and
 
 \item families of the form
\begin{equation} \label{equa:cartCov} \tag{Car}
\{(Y' \times_Y U /Y') \to (U/Y)\} 
\end{equation}
for morphisms $Y'\to Y$ in $\Mod_X$.
%
\end{enumerate}
We will write $\Shv(\RZ(X_{\Nis}))$ for the topos associated to the topology generated by coverings of the form \eqref{equa:RZNisCov} and \eqref{equa:cartCov}.
\end{defi}


\begin{rema} \label{rema:carLoc}
Since the diagonal of a modification is again a modification, a presheaf (of sets) satisfies descent for all families \eqref{equa:cartCov} if and only if it sends each $(Y' \times_Y U /Y') \to (U/Y)$ to an isomorphism. Consequently, 
\[ \Shv_{\car}(\RZ(X_{\Nis})) = \lim_{Y \in \Mod_X} \PSh(Y_{\Nis}), \]
where the limit is along pushforwards $f_*: \PSh(Y'_{\Nis}) \to \PSh(Y_{\Nis})$ for morphisms $f: Y' \to Y$ in $\Mod_X$. This implies
\begin{equation}
\Shv(\RZ(X_{\Nis})) = \lim_{Y \in \Mod_X} \Shv_{\Nis}(Y_{\Nis}).
\end{equation}
\end{rema}

\begin{rema} \label{rema:modToEquivSpaces}
The same is true for presheaves of spaces. Suppose that $F$ satisfies descent for families of the form \eqref{equa:cartCov}. If $Y' \to Y$ in $\Mod_X$ is a closed immersion then $(Y')^{\times_Y n} = Y'$ so for any $(U/Y)$ in $\RZ(X_{\Nis})$ we have 
\begin{equation} \label{equa:UYUY}
F(U/Y) = \lim_n F((U'/Y')^{\times_{(U/Y)} (n+1)}) = \lim_n F(U'/Y') = F(U'/Y')
\end{equation}
 where $U' = Y' \times_Y U$. It follows that \eqref{equa:UYUY} also holds for a general $Y' \to Y$ in $\Mod_X$, since each diagonal $Y' \to (Y')^{\times_Y n}$ is a closed immersion in $\Mod_X$. Conversely, if $F$ sends families of the form \eqref{equa:cartCov} to equivalences, then it clearly satisfies \v{C}ech descent for such families.
\end{rema}

%


\begin{prop} \label{prop:sheafificationCalc}
Let $X$ be a qcqs scheme and suppose $F \in \PSh(\RZ(X_{\Nis}))$ has descent for the coverings \eqref{equa:RZNisCov}. Then the sheafification $aF \in \Shv(\RZ(X_{\Nis}))$ satisfies
\begin{equation} \label{equa:aFUY}
aF(U/Y) = \colim_{Y' \in (\Mod_X)_{/Y}} F(Y' \times_Y U / Y').
\end{equation}
The same is true for presheaves of spaces.
\end{prop}

\begin{proof}
%
%
First we show that the presheaf $aF$ defined via \eqref{equa:aFUY} is a sheaf. By definition, a presheaf on $\RZ(X_{\Nis})$ is a sheaf if and only if it has descent for coverings of the form \eqref{equa:RZNisCov} and \eqref{equa:cartCov} in Definition~\ref{defi:RZtop}. %
The presheaf $aF$ in the statement certainly sends modifications to isomorphisms, resp. equivalences, so it has descent for coverings of the form \eqref{equa:cartCov} by remarks \ref{rema:carLoc} and \ref{rema:modToEquivSpaces}. 

For Nisnevich coverings, we notice that a presheaf $F$ has descent for coverings of the form \eqref{equa:RZNisCov} if and only if the restriction to the small Nisnevich site $Y_{\Nis}$ for each $Y \in \Mod_X$ has Nisnevich descent if and only if it sends distinguished Nisnevich squares to cartesian squares. If $\{U_0 \to U, U_1 \to U\}$ is a distinguished Nisnevich square, and $(U/Y) \in \RZ(X_{\Nis})$ then for any $Y' \to Y$ in $\Mod_X$ we have 
\[ F(Y' \times_Y U / Y') = F(Y' \times_Y U_0 / Y') \times_{F(Y' \times_Y U_{01} / Y')} F(Y' \times_Y U_1 / Y') \]
where $U_{01} = U_0 \times_U U_1$ by the assumption that $F$ has descent for \eqref{equa:RZNisCov}. Taking the colimit over $Y'$ and using the fact that filtered colimits commute with fibre products we find 
\[ aF(U / Y) = aF(U_0 / Y) \times_{aF(U_{01} / Y)} aF(U_1 / Y). \]

So $aF$ is a sheaf. To conclude that $a$ is the sheafification functor, it suffices to show that if $F$ is already a sheaf, then $F \to aF$ is an equivalence. But this is clear, since sheaves send modifications to isomorphisms, Rem.\ref{rema:carLoc}, resp.\ equivalences, Rem.\ref{rema:modToEquivSpaces}.
%
\end{proof}

\begin{defi}
Let $S$ be a qcqs scheme. For $X\in \Sch_S$, we consider the canonical projection functor
\[ \rho_X: \RZ(X_{\Nis}) \to \Sch_S; \quad (U/Y) \mapsto U \]
and the functor induced by composition
\[ \PSh(\Sch_S) \to \PSh(\RZ(X_{\Nis})); \qquad F \mapsto F \circ \rho_X. \]
By composing this with the sheafification functor $\PSh(\RZ(X_{\Nis})) \to \Shv(\RZ(X_{\Nis}))$, we get  
\begin{equation}\label{rhoX}
\rho^*_X: \Shv_{\procdh}(\Sch_S) \stackrel{}{\to} \Shv(\RZ(X_{\Nis})). 
\end{equation}
\end{defi}

\begin{rema}\label{rema;rhoX}
Using Proposition~\ref{prop:sheafificationCalc} we have the following concrete description.
\[ (\rho^*_X F)(U/Y) = \colim_{Y' \in \Mod_X} F(Y' {\times}_X U). \]
\end{rema}

Recall that a morphism of sites $\phi: C \to D$ is \emph{cocontinuous} if for every $U \in C$ and covering family $\sU = \{U_i \to \phi U \}_{i \in I}$ there is a covering family $\{V_i \to U_i\}$ such that $\{\phi V_i \to \phi U\}_{i \in I}$ refines $\sU$, \cite[Def.III.2.1]{SGA41}, \cite[00XJ]{stacks-project}.

\begin{prop} \label{prop:cocont}
Let $X$ be a qcqs scheme whose underlying topological space is Noetherian.
Then the morphism $\rho$ is cocontinuous. 
%
%
\end{prop}

\begin{proof}
For $(U/Y) \in \RZ(X_{\Nis})$ and a procdh covering $\{V_i \to U\}_{i \in I}$ in $\Sch_S$, we want to find a covering $\{(U_j/Y_j) \to (U/Y)\}_{j \in J}$ in $\RZ(X_{\Nis})$, 
a function $J \to I$; $j \mapsto i_j$, and commutative triangles
\[ \xymatrix{
U_j \ar@{-->}[r] \ar[dr] & V_{i_j} \ar[d] \\
& U.
} \]

Since procdh coverings are refined by finite length compositions of generator procdh coverings, using induction on the height of the tree in Remark~\ref{rema:genDefi}, it suffices to prove the claim for distinguished Nisnevich coverings and proabstract blowup squares. 


For Nisnevich coverings the statement is obvious since for any $(U/Y) \in \RZ(X_{\Nis})$, a Nisnevich covering $\{U_i \to U\}_{i \in I}$ of $U$ gives rise to a Nisnevich covering $\{
(U_i/Y) \to (U/Y)\}_{i \in I}$ of $(U/Y)$. 

Consider $(U/Y) \in \RZ(X_{\Nis})$ and a proabstract blowup square $\sU = \{Z_n \to U\}_{n \in \NN} \sqcup \{W \to U\}$. Our task is to find a morphism 
$Y' \to Y$
in $\Mod_X$
and a Nisnevich covering
$\{V_{j} \to U' := Y' \times_{Y} U\}_{j \in J}	$ 
such that for each $j$, we have a commutative diagram on the left for some $n$ or on the right.
\begin{equation} \label{equa:VZUVWU}
\xymatrix@!R=3pt{
V_j \ar[d] \ar[r] & Z_n \ar[d] && V_j \ar[r] \ar[d] & W \ar[d] \\
U' \ar[r] \ar[d] & U \ar[d] && U' \ar[r] \ar[d] & U \ar[d] \\
Y' \ar[r] & Y && Y' \ar[r] & Y
}
\end{equation}
%
Since $U$ has finitely many generic points, by Lemma~\ref{lemm:RZHausdorff}, we can assume that $U$ is locally irreducible. As such it suffices to treat the following two cases.

Case 1: $U^\gen \subseteq Z_0$.
In this case, $(Z_{0})_{\red} = U_{\red}$. This means the (finitely many) generators of $\sI_{Z_{0}}$ are nilpotent, so $Z_{n} = U$ for some $n$. Hence, $Y' = Y$ and the trivial covering $\{U \to U\}$ give a square on the left of \eqref{equa:VZUVWU}.

Case 2: $U^\gen \cap Z_0 = \varnothing$.
We will build a square as on the right of \eqref{equa:VZUVWU} with $V_j = U'$. %
By the assumption $U^\gen \cap Z_0 = \varnothing$, the morphism $W \to U$ is an isomorphism over the generic points of $U$. These all lie over generic points of $Y$ (because $U \to Y$ is étale) so $W \to U \to Y$ is generically flat. 
More precisely, letting $T\subset Y$ be the closure of the image of $Z_0$ in $Y$, 
the morphism $W\to Y$ is flat over $Y\backslash T$ and $T$ is nowhere dense in $Y$ by the assumption $U^\gen \cap Z_0 = \varnothing$. 
Write $T=\varprojlim_{\lambda} T_\lambda$ for a cofiltered system of closed immersions $T_\lambda\to T$ of finite presentation. Since $Y^{\gen}$ is finite by the assumption, there exists $\lambda$ such that $T_\lambda$ is nowhere dense.

By Raynaud-Gruson \cite[Th.5.2.2]{RG71}, \cite[081R]{stacks-project}, there is a blowup $Y' \to Y$ (projective, but not necessarily of finite presentation) with a center contained in $T_\lambda$ %
such that the strict transform $W' \to Y'$ of $W \to Y$ is flat. Since $U' := Y' \times_Y U \to Y'$ is étale, this implies that $W' \to U'$ is flat, Lemma~\ref{lemm:etaleFlat}.
  So now we have a flat proper morphism which is generically an isomorphism. This implies it is globally an isomorphism, Lemma~\ref{lemm:properFlat}.
 So we obtain a factorisation $U' \cong W' \to W \to U$. 

If $Y' \to Y$ is not of finite presentation, then we can at least write it as a filtered limit $Y' = \lim Y_λ'$ such that each $Y'_λ \to Y$ is in $\Mod_Y$ and the transition morphisms are affine, Lemma~\ref{lemm:rema:genXY}.
 Setting $U'_λ = Y_λ' \times_Y U$ we also have $U' = \lim U_λ'$ and we can descend $U' \to W$ to some $U'_λ$, \cite[Prop.8.13.1]{EGAIV3}.
Noting that $Y'_λ \in \Mod_Y$ implies $Y'_λ \in \Mod_X$, we have found a square on the right of \eqref{equa:VZUVWU}.
\end{proof}

\begin{coun}
Let $k$ be a field, $I$ a set, and $k[x^I]$
the polynomial ring in $I$-many variables. Define $R^I := k[x^I] / \langle x_ix_j: i \neq j \rangle$ and $\XX^I = \Spec R^I$. So $\XX^I$ is $I$-many copies of the affine line joined at the origin and $(\XX^I)^\gen$ is homeomorphic to $I$ with the discrete topology. Given a subset $J \subseteq I$ there is an associated closed immersion $\XX^J \subseteq \XX^I$ defined by $x_i \mapsto 0$ for $i \notin J$ which is finite presentation if and only if $I \setminus J$ is finite. 

For any two cofinite $J, J' \subseteq I$ with $J \cup J' = I$, the set $\{\XX^J \to \XX^I, \XX^{J'} \to \XX^I\}$ induces a proabstract blowup square. This cannot be refined by a modification (of finite presentation), so the functor $\RZ(\XX^I_{\Nis}) \to \Sch_{\XX^I}$ is \emph{not} cocontinuous. 

Indeed, by Chevalley's Theorem \cite[054K]{stacks-project}, the image of any summand $Y_0$ of a modification $Y_0 \amalg Y_1 \to \XX^I$ is a closed subscheme of finite presentation. Since modifications are isomorphisms over a dense open, this implies that for any modification $Y \to \XX^I$ the scheme $Y$ is connected. Since it is an isomorphism over the generic points of $\XX$, it cannot factor through any $\XX^J \to \XX^I$ unless $J = I$.
\end{coun}


\begin{coro} \label{coro:RZColimitLim}
Let $S$ be a qcqs scheme whose underlying topological space is Noetherian and $X \in \Sch_S$.
The canonical functor $\rho_X^*$ from \eqref{rhoX}
preserves colimits and finite limits, \cite[III.2.3]{SGA41}, \cite[00XL]{stacks-project}.
\end{coro}

Here are some lemmas that were used above.

\begin{lemm} \label{lemm:rema:genXY}
Suppose $X$ is a qcqs scheme. 
\begin{enumerate}
 \item \label{lemm:rema:genXY:bir} If the underlying topological space of $X$ is Noetherian, then for every $Y \in \Mod_X$ there is a closed immersion $Y' \to Y$ in $\Mod_X$ such that $(Y')^\gen = X^\gen$.
 \item If $Y \to X$ is a proper morphism, not necessarily of finite presentation, which is an isomorphism over a dense qc open $D \subseteq X$, then $Y$ can be written as a cofiltered limit $Y = \lim Y_λ$ with $Y_λ \in \Mod_X$ and whose transition morphisms are closed immersions $Y_λ \to Y_μ$ which are isomorphisms over $D$.
\end{enumerate}
\end{lemm}

\begin{proof}
\begin{enumerate}
 \item Let $D \subseteq X$ be a dense qc open over which $Y \to X$ is an isomorphism. Let $j:D \to Y$ be the induced open immersion. Write $\sI = \ker(\OO_Y \to j_*\OO_D)$ as a filtered colimit $\sI = \colim \sI_λ$ of ideals $\sI_λ$ of finite presentation, \cite[Cor.6.9.15]{EGAI}. Since $\uSpec (\OO_Y / \sI) \to Y$ is an isomorphism over the open $D \subseteq Y$, there is some $λ$ for which $Y'_λ := \uSpec (\OO_Y / \sI_λ) \to Y$ is an isomorphism over $D \subseteq Y$, \cite[Thm.8.10.5(i)]{EGAIV3}. Since $X$ has Noetherian topological space, so does $Y$, Lem.\ref{lemm:noethTopFt}, so up to changing $λ$, we can assume $Y'_λ \cap (Y^\gen \setminus X^\gen) = \varnothing$, i.e., $(Y'_λ)^\gen = X^\gen$.

\item We can write $Y \to X$ as $Y = \lim Y_λ$ with each $Y_λ \to X$ proper and of finite presentation and all transition morphisms closed immersions by \cite[09ZR, 09ZQ]{stacks-project}. 
The isomorphism $D \times_X Y \cong D$ induces closed immersions $D \to D \times_X Y_λ$ which are of finite presentation by \cite[00F4(4)]{stacks-project}, and therefore defined by coherent sheaves of ideals, \cite[01TV]{stacks-project}. Since the $D \times_X Y_λ$ are all quasi-compact, there is a $λ$ for which $D = D \times_X Y_μ$ for all $\mu \geq λ$.
\end{enumerate}
\end{proof}

\begin{lemm} \label{lemm:etaleFlat}
Suppose that $W \to U$ is any morphism of schemes, $U \to Y$ is étale and $W \to Y$ is flat. Then $W \to U$ is also flat.
\end{lemm}

\begin{proof}
If $U \to Y$ is not already assumed to be separated, we can reduce to this case by replacing $U$ with an open affine covering. Then the diagonal $δ: U \to U \times_Y U$ is open (resp. closed) because $U \to Y$ is unramified (resp. separated). Considering the cartesian squares
\[ \xymatrix{
W \ar[r]^-{\delta'} \ar[d] & U \times_Y W \ar[r]^-{pr_2} \ar[d] & W \ar[d] \\
U \ar[r]^-\delta & U \times_Y U \ar[r]^-{pr_2} & U
} \]
one sees that $δ': W \to U \times_Y W$ is also both open and closed. So we have $U \times_Y W = δ'(W) \sqcup W'$ for some $W'$. Since $W \to U$ factors as the composition $W \to δ'(W) \sqcup W' = U \times_Y W \to U$ of the inclusion of a summand, and the pullback $U \times_Y W \to U$ of the flat morphism $W \to Y$, it is flat.
\end{proof}

\begin{lemm} \label{lemm:properFlat}
Suppose that $f: W \to U$ is a flat, proper morphism of schemes, and $D \subseteq U$ is a schematically dense open such that $D \times_U W \to D$ is an isomorphism. Then $W \to U$ is an isomorphism.
\end{lemm}

\begin{proof}
First note that $f$ is faithfully flat: Since it is proper, it is closed, so it suffices to show all generic points $η$ of $U$ are in the image. The inclusion of the localisation $U_η \to U$ is flat so $U_η \times_U D \to U_η$ is schematically dense. But $U_η$ has a single point, so $U_η \times_U D$ is non-empty, i.e., $η \in D$, so $η \in f(W)$.

Since $W \to U$ is faithfully flat, to show it is an isomorphism it suffices to show that $pr_1: W \times_U W \to W$ is an isomorphism, or equivalently, that the diagonal $δ: W \to W \times_U W$ is an isomorphism. 
The two morphisms $W \times_U W \stackrel{pr_1}{\to} W \to U$ are flat, so $D \times_U W \to W$ and $D \times_U W \times_U W \to W \times_U W$ are schematically dense. But $W \to U$ is an isomorphism over $D$ so $D \cong D \times_U W \cong D \times_U W \times_U W$. So we have a factorisation $D \stackrel{j}{\to} W \stackrel{δ}{\to} W \times_U W$ with both $j$ and $δj$ schematically dense. It follows that $δ$ is schematically dense. Since it is also a closed immersion ($f$ is proper, so separated) this implies $\delta$ is an isomorphism.
\end{proof}

\begin{lemm} \label{lemm:RZHausdorff}
Suppose $Y$ is qcqs with Noetherian underlying space, let $U \to Y$ be an étale morphism, and suppose $U^{\gen} = η_0 \sqcup η_1$ is a decomposition of the space of generic points of $U$ into clopens. Then there exists a \emph{cartesian} square
\[ \xymatrix{
U'_0 \sqcup U_1' \ar[d] \ar[r] & U \ar[d] \\
Y' \ar[r] & Y
} \]
such that
\begin{enumerate}
	\item $Y' \to Y$ is a proper morphism of finite presentation which is an isomorphism over a dense qc open of $Y$, and
	\item there are identifications $η_0 = (U_0')^{\gen}$ and $η_1 = (U_1')^{\gen}$.
\end{enumerate}
\end{lemm}

\begin{proof}
Choose any closed subschemes 
$U_0 \subseteq U$ 
and 
$U_1 \subseteq U$ 
of finite presentation whose generic schemes are $η_0$ and $η_1$ respectively (here we are using that $U^\gen$ is finite).%
\footnote{For example, let $\sI_Z = \ker(\OO_U \to \OO_{U_Z^{\gen}})$. Write $\sI_Z$ as a filtered union of ideals $\sI_{Z,\lambda}$ of finite presentation. Use the fact that $U^\gen$ is a disjoint union of finitely many local schemes of dimension zero to deduce that there is some $\sI_\lambda$ with $(\uSpec \OO_U / \sI_{\lambda}) \cap U^\gen_W = \varnothing$.}  %
Since $U_0 \sqcup U_1 \to U \to Y$ is generically étale, there is some dense open $D \subseteq Y$ over which it is étale, \cite[07RP]{stacks-project}, %
and in particular, flat and of finite presentation. %
So we can apply Raynaud-Gruson platification, \cite[081R]{stacks-project}, to find a blowup $Y' \to Y$ (not necessarily of finite presentation) %
which is an isomorphism over $D \subseteq Y$, and for which the strict transform $U_0' \sqcup U_1' \to Y'$ of $U_0 \sqcup U_1 \to Y$ is flat. %
By Lemma~\ref{lemm:etaleFlat} this implies $U_0' \sqcup U_1' \to Y' \times_Y U$ is also flat. It is proper and an isomorphism over a dense open by construction, so it is in fact an isomorphism, Lem.\ref{lemm:properFlat}. Now the second condition is satisfied, since $U_0' \to U$ factors through $U_0' \to U_0$ and this latter is an isomorphism generically by construction.

The morphism $Y' \to Y$ might not be of finite presentation, but it is a cofiltered limit of some $Y_λ' \to Y$ which are proper, of finite presentation, and isomorphisms over dense qc opens $D_λ \subseteq Y$, Lem.\ref{lemm:rema:genXY}. There is some $λ$ for which we already have a decomposition $U_λ' = U_{λ,0}' \sqcup U_{λ,1}'$ where of course we have set $U_λ' = Y_λ \times_Y U$ and $U_{λ,i}' = Y'_λ \times_Y U_i$. %
Replacing $Y'$ with this $Y'_λ$, both conditions are now satisfied.
\end{proof}

\begin{lemm} \label{lemm:noethTopFt}
Suppose that $X$ is a scheme with Noetherian topological space of finite dimension and $Y \to X$ is a morphism of finite type. Then $Y$ also has Noetherian topological space of finite dimension.
\end{lemm}

\begin{proof}
We want to show that every descending chain 
\[ Z_0 \supseteq Z_1 \supseteq Z_2 \supseteq \dots \]
of closed subspaces of $Y$ stabilises. Since $X$ has Noetherian topological space, it admits a finite open affine covering $\{\Spec(A_i)\}_{i = 1}^n$. It suffices to show that our chain stabilises in each $\Spec(A_i) \times_X Y$, that is, we can assume $X$ is affine.

Similarly, by definition, since $Y \to X$ is finite type it is quasi-compact, so we can also assume $Y$ is affine, say $Y = \Spec(B)$. In particular, this means there is a closed immersion $Y \to \AA^n_X$ for some $n$. So we can assume that $Y = \AA^n_X$. By induction we can assume $n = 1$. 

Since $X$ has Noetherian topological space, it has finitely many irreducible components. 
So we can assume that $X$ is integral. 

Now we use induction on the dimension of $X$. If $X$ is dimension -1 it is empty, and we are done (the dimension zero case is also easy).

In general, suppose that 
\[ \AA^1_X \supseteq Z_0 \supseteq Z_1 \supseteq  Z_2 \supseteq \dots \]
is a decreasing chain of closed subschemes. I claim that it suffices to show that chains of finite presentation closed subschemes stabilise. Indeed, if there exists a strictly decreasing chain of closed subsets, then we can manufacture a strictly decreasing chain of closed subsets of finite presentation: For each $i$ choose a prime $\p_i \in Z_i \setminus Z_{i+1}$ and an element $f_{i+1} \in I_{i+1} \setminus \p_i$ where $I_{i+1}$ is the ideal corresponding to $Z_{i+1}$. Then setting $W_i$ to be the closed corresponding to $\langle f_1, \dots, f_i \rangle$, the chain $W_0 \supseteq W_1 \supseteq W_2 \supseteq \dots$ contains $Z_0 \supseteq Z_1 \supseteq \dots$ as a subchain and also has $\p_i \in W_i \setminus W_{i+1}$, so it is strictly decreasing.

So we can assume $Z_n = \langle f_1, \dots, f_n \rangle$ for some $f_i \in A[x]$ where $A = \Gamma(X, \OO_X)$. Letting $K$ be the quotient field of $A$, the ring $K[x]$ is Noetherian, so the corresponding chain of ideals (not just closed subsets) in $K[x]$ stabilises. Moreover, $K[x]$ is a principal ideal domain, so there is some $a \in A \setminus \{0\}$ and $g(x) \in A[x]$ such that $\langle \tfrac{1}{a}g \rangle = \langle f_1, \dots, f_n \rangle$ for all $n \gg 0$. That is, we can write each $f_i$ as $f_i = \tfrac{1}{a} g \tfrac{1}{b_i} h_i$ for some $h_i \in A[x]$ and $b_i \in A \setminus \{0\}$, and $1 = \sum \tfrac{1}{b_i} h_i \tfrac{1}{c_i} k_i$ for some $k_i \in A[x]$ and $c_i \in A \setminus \{0\}$ (for simplicity, we choose $k_j = 0$ $c_j = 1$ for $j \gg 0$ so that the $k_j$ and $c_j$ are independent of $n$). If the $a, b_i, c_i$ are all 1, then $\langle g \rangle = \langle f_1, \dots, f_n \rangle$ holds in $A[x]$, not just $K[x]$. That is, our chain stabilises. Indeed, $f_i = gh_i$ and $g = g(\sum h_ik_i) = \sum f_ik_i$.

We now show how to reduce to the case all $a, b_i, c_i$ are 1. Set $a_i = a\prod b_i \prod c_i$. Since $\Spec(A)$ has Noetherian topological space, the chain of opens $\Spec(A[\tfrac{1}{a_0}]) \subseteq \Spec(A[\tfrac{1}{a_0a_1}]) \subseteq \Spec(A[\tfrac{1}{a_0a_1a_2}]) \subseteq \dots $ stabilises. That is, there is some $d \in A$ such that $A[d^{-1}] = A[\tfrac{1}{a_0a_1\dots a_n}]$ for all $n \gg 0$. Since $d$ is nonzero and $A$ is integral, $\Spec(A/\langle d \rangle)$ has dimension strictly smaller than $\Spec(A)$, so our chain of $Z_i$'s restricted to $\AA^1_{\Spec(A/\langle d \rangle)}$ stabilises by the induction hypothesis, and it suffices to show that it stabilises when restricted to $\AA^1_{\Spec(A[d^{-1}])}$. By construction, $a, b_i, c_i$ are all invertible in $A[d^{-1}]$. So replacing $A$ with $A[d^{-1}]$, we can assume they are all 1 by replacing $g, h_i, k_i$ with $\tfrac{1}{a}g$, $\tfrac{1}{b_i}h_i$ and $\tfrac{1}{c_i}k_i$.
\end{proof}

\section{Conservativity of the fibre functors} \label{sec:enoughPoints}

Recall that a topos is said to have \emph{enough points} when a morphism $f$ is an isomorphism if and only if $\phi(f)$ is an isomorphism for all fibre functors $\phi$, Def.\ref{defi:FF}, \cite[Exposé IV, Déf.6.4.1]{SGA41}. Equivalently, \cite[Exposé IV, Prop.6.5(a)]{SGA41}, a topos of the form $\Shv_\tau(C)$ has enough points when a family $\{Y_i \to X\}_{i \in I}$ in $C$ is a covering family if and only if $\sqcup_{i \in I} \phi(Y_i) \to \phi(X)$ is surjective for all fibre functors $\phi$.\footnote{Here, we have used the same symbol for an object $X$ of $C$ and the sheafification of the presheaf $\hom(-, X)$ it represents.}

Deligne's completeness theorem says that if $C$ is an essentially small category with fibre products, and every $\tau$-covering is refinable by a finite one, then $\Shv_\tau(C)$ has enough points, \cite[Prop.VI.9.0]{SGA42} or \cite[Thm.7.44, 7.17]{Joh77}. 

\begin{exam} \label{exam:RZNisenoughPoints}
If $X$ is a qcqs scheme then $\Shv(\RZ(X_{\Nis}))$ has enough points.
\end{exam}

%

\begin{theo} \label{theo:procdhEnoughPoints}
Suppose $S$ is a qcqs scheme with Noetherian topological space of finite Krull dimension. 
Then the topos $\Shv_{\procdh}(\Sch_S)$ has enough points.
\end{theo}

\begin{proof}
Suppose that $\cY = \{Y_i \to Y\}_{i \in I}$ is a family of morphisms in $\Sch_S$ such that the morphism of sets
$\sqcup_i \phi(Y_i) \to \phi(Y)$ 
is surjective for every fibre functor $\phi$. We want to show that $\cY$ is refinable by a $\procdh$-covering. We work by induction on the Krull dimension of $Y$, the base case being $Y = \varnothing$ with $\dim Y = -1$. In this base case, either $I$ is empty, or $I$ is nonempty and each $Y_i \to Y$ is an isomorphism. Both of these are already covering families, so no refinement is necessary.

Now we do the induction step. The functor
$\rho^*_Y:\Shv_{\procdh}(\Sch_S) \stackrel{}{\to} \Shv\RZ(Y_{\Nis})$ from \eqref{rhoX} preserves colimits and finite limits, Corollary \ref{coro:RZColimitLim}.
So by composition, every fibre functor $\phi$ of $\Shv\RZ(Y_{\Nis})$ induces a fibre functor 
\[\Shv_{\procdh}(\Sch_S) {\stackrel{\rho_Y^*}{\to}} \RZ(Y_{\Nis}) \stackrel{\phi}{\to} \Set.\]

 By assumption, our family $\cY$ is sent to a surjection of sets under each such fibre functor $\phi \circ \rho_Y^*$. Since $\RZ(Y_{\Nis})$ has enough points, Exam.\ref{exam:RZNisenoughPoints}, it follows that $\rho_Y^*\cY$ is a surjective family of sheaves in $\Shv(\RZ(Y_{\Nis}))$. This means that, locally, we can lift the section $\id_{Y}$ of $(\rho_Y^*Y)((Y/Y)) = \hom_{\Sch_S}(\rho((Y/Y)), Y) = \hom_{\Sch_S}(Y, Y)$. Explicitly, this means that there exists a covering $\{(U_j/Y') \to (Y/Y)\}_{j \in J}$ such that the family $\{U_j \to Y' \to Y\}_{j \in J}$ refines $\cY$. %
Since $Y' \to Y$ is a modification, there is a nowhere dense closed subscheme of finite presentation $Z_0 \to Y$ outside of which $Y' \to Y$ is an isomorphism. %
Since $Y$ has finite Krull dimension and $Z_0 \to Y$ is nowhere dense, $\dim Z_0 < \dim Y$. %
So by the induction hypothesis, the pullbacks $Z_n \times_Y \cY$ of $\cY$ to each $Z_n$ also admit refinements by $\procdh$-coverings (here we are using that fibre functors preserve finite limits to know that each $\phi(Z_n \times_Y \cY)$ is a surjective morphism of sets). Composing all these $\procdh$ coverings produces a $\procdh$-covering of $Y$ which refines the original $\cY$.
\end{proof}

%

\begin{coro} \label{coro:smallFamily}
Suppose $S$ is a Noetherian scheme of finite Krull dimension. 
Then procdh local $S$-rings $R$ with $\length Q(R)$ and $\dim R$ finite induce a conservative family of fibre functors of $\Shv_{\procdh}(\Sch_S)$.
\end{coro}

\begin{proof}
Apply Proposition~\ref{prop:colimOfSmaller}.
\end{proof}

\section{Excision} \label{sec:excision}

In this section we discuss the equivalence of procdh excision and \v{C}ech descent. There is nothing particularly new in this section (apart from the results) but we include complete proofs for convenience and because the statements in the literature don't quite fit our situation.

\emph{Sheaves of spaces.}
Since \cite{HTT}, the default has become to define a \emph{sheaf of spaces} as a presheaf of spaces satisfying \v{C}ech descent.\footnote{As opposed to the hyperdescent used in work of 
Artin, Brown, Deligne, Friedlander, Gersten, Jardine, Joyal, Mazur, Morel, Thomason, Verdier, Voevodsky, and many others in work dating back at least to the 70's.
} 
That is, a presheaf of spaces such that 
every $τ$-covering $\{Y_i \to X\}_{i \in I}$ we have an equivalence
\begin{equation} \label{equa:cech}
F(X) = \lim_n \prod_{(i_0, \dots, i_n) \in I^{n+1}} F(Y_{i_0} \times_X \dots \times_X Y_{i_n}).
\end{equation}
in the category of spaces.
Equivalently, a $τ$-sheaf is a presheaf such that
\[ \Map(X, F) \to \Map(R, F) \]
is an equivalence for every $τ$-covering sieve $R \subseteq \Map(-, X)$. This ``equivalently'' is essentially because, writing $\cY_n := \coprod_{i \in I^{n+1}} \Map(-, Y_{i_0} \times_X \dots \times_X Y_{i_n})$, 
\begin{equation} \label{equa:cech2}
\cY_0
\stackrel{p}{\to}
\colim_n \cY_n = R
\stackrel{j}{\to} %
\Map(-, X)
\end{equation}
is the canonical factorisation in $\PSh(C, \cS)$ such that for all $T \in C$, the morphism of spaces $p(T)$ is surjective on connected components and the morphism of spaces $j(T)$ is an inclusion of connected components, \cite[Prop.6.2.3.4]{HTT}. That is, $R$ is the sieve associated to $\{Y_i \to X\}_{i \in I}$.

\begin{theo} \label{theo:descentConditions}
Let $S$ be a scheme and consider the following conditions on a presheaf of spaces $F \in \PSh(\Sch_S, \cS$). 
\begin{enumerate}
 \item
 \textnormal{
 \emph{Excision.} 
 For every distighished Nisnevich square $\{U \to X, V \to X\}$ and every proabstract blowup square $\{Z_n \to X\}_{n \in \NN} \sqcup \{Y \to X\}$ in $\Sch_S$, Def.\ref{defi:procdh}, we have
\begin{align}
F(X) &\stackrel{\sim}{\to} F(U) \times_{F(U \times_X V)} F(V), %
\\
F(X) &\stackrel{\sim}{\to} F(Y)\times_{\lim F(Z_n \times_X Y)} \lim F(Z_n).
\end{align}
}

 \item
 \textnormal{
 \emph{\v{C}ech descent.} 
 For every procdh covering family $\{Y_λ \to X\}_{λ \in \Lambda}$ we have 
 \begin{equation}
 F(X) \stackrel{\sim}{\to} \lim_{n} 
F(\cY_n)
 \end{equation}
where we write $F(\cY_n)$ for $\prod_{i \in I^n} F(Y_{i_1} \times_X \dots \times_X F_{i_n})$.
That is, in the terminology of \cite{HTT}, $F$ is a procdh sheaf.
}

\end{enumerate}
If $S$ is qcqs then (Excision) $\Leftrightarrow$ (\v{C}ech descent). 
\end{theo}

The implication (Excision) $\Rightarrow$ (\v{C}ech descent) is Proposition~\ref{prop:excisionCech} below, and
the implication (Excision) $\Leftarrow$ (\v{C}ech descent) is Proposition~\ref{prop:hyperExci} below. 

\begin{rema}
	The equivalence (Excision) $\Leftrightarrow$ (\v{C}ech descent) is essentially due to Voevodsky, \cite{Voe10}, building on work of Brown and Gersten, \cite{BG73}. Our proofs are essentially the ones from \cite[Thm.3.2.5]{AHW15}, although the techniques are in \cite[Exp.Vbis \S 3.3]{SGA42}, and the key points for (Excision) $\Leftarrow$ (\v{C}ech descent)---that the horizontal morphisms are categorical monomorphisms and the diagonal squares are distinguished---are isolated in \cite[Def.2.10]{Voe10}.
\end{rema}

\begin{prop} \label{prop:excisionCech}
Suppose that $S$ is a qcqs scheme and $F \in \PSh(\Sch_S, \cS)$ a presheaf of spaces. 
If $F$ satisfies procdh excision then $F$ satisfies procdh \v{C}ech descent, cf.Thm.\ref{theo:descentConditions}.
\end{prop}

\begin{rema}
If we restrict our attention to presheaves of sets, then \v{C}ech descent is just the usual sheaf condition.
\end{rema}

\begin{proof}
Suppose that $F$ satisfies excision and $\{Y_i \to X\}_{i \in I}$ is a procdh covering. By definition, every covering is refinable by a composition of finite length of generator coverings, cf. Remark~\ref{rema:genDefi}. We work by induction on the length of the composition, that is, on the height of the tree from  Remark~\ref{rema:genDefi}. 

\emph{Step 0.} The case of length zero is when there exists a factorisation $X \to Y_i \to X$ for some $i$. This induces a section 
\[ X \to \colim_n \sqcup_{i \in I^{n+1}} Y_{i_0} \times_X \dots \times_X Y_{i_n} \stackrel{j}{\to} X \]
 in $\PSh(\Sch_S, \cS)$ where we have written just $T$ for $\Map(-, T)$. Since $j$ is a categorical monomorphism, \cite[Prop.6.2.3.4]{HTT}, existence of a section forces it to be an equivalence. Hence, the morphism $\Map(j, F)$ is an equivalence. But this is none-other-than $F(X) \to \lim_n \prod_{i \in I^{n+1}} F(Y_{i_0} \times_X \dots \times_X Y_{i_n})$ being an equivalence.

\emph{Step $(n+1)$.}  Suppose $\{Y_i \to X\}_{i \in I}$ is refinable by a distinguished Nisnevich covering $\{U \to X, V \to X\}$. 
Note that this means that for $T = U, V, W$ the covering $\{Y_i \times_X T \to T\}_{i \in I}$ is the case of length zero we just considered.
Setting $W = U \times_X V$ and $F(Y_{n,T}) = \prod_{i \in I^{n+1}} F(Y_{i_0} \times_X \dots \times_X Y_{i_n} \times_X T)$ we get 
\begin{align*}
\lim_n F(Y_n) 
&=
\lim_n  \biggl ( F(Y_{n,U}) \times_{F(Y_{n,W})} F(Y_{n,V}) \biggr ) 
& (\textrm{Excision}) 
\\
&= 
\lim_n F(Y_{n,U}) \times_{\lim_n F(Y_{n,W})} \lim_n F(Y_{n,V}) 
& (\textrm{Limits commute with limits}) 
\\ &=
F(U) \times_{F(W)} F(V) 
& (\textrm{Step 0}) 
\\ 
&=
F(X).
& (\textrm{Excision})
\end{align*}
A similar argument works for a procdh covering. Explicitly, if $\{Z_n \to X\}_{n \in \NN} \sqcup \{B \to X\}$ is a proabstract blowup square refining $\{Y_i \to X\}_{i \in I}$, and $E_n := Z_n \times_X B$ we have 
\begin{align*}
\lim_n F(Y_n) 
&=
\lim_n  \biggl ( \lim_m F(Y_{n,Z_m}) \times_{\lim_m F(Y_{n,E_m})} F(Y_{n,B}) \biggr ) \\
&= 
\lim_m \lim_n  F(Y_{n,Z_m}) \times_{\lim_m \lim_n  F(Y_{n,E_m})} \lim_n  F(Y_{n,B})  \\
&= \lim_m F(Z_m) \times_{\lim_m F(E_m)} F(B) \\ 
&= F(X).
\end{align*}
\end{proof}

\begin{prop} \label{prop:hyperExci}
Suppose that $S$ is a qcqs scheme and $F \in \PSh(\Sch_S, \cS)$ a presheaf of spaces. 
If $F$ satisfies procdh \v{C}ech descent then it satisfies procdh excision, cf.Thm.\ref{theo:descentConditions}.
\end{prop}

\begin{proof}
Consider the two squares 
\begin{equation} \label{equa:NisPorcdhSq}
\xymatrix{
W \ar[r] \ar[d] & V \ar[d] & \ar@{}[d]|{\textstyle{\textnormal{ and }}}& ``\colim"_n E_n \ar[r] \ar[d] \ar@{}[dr]|{(\ast)} & Y \ar[d] \\
U \ar[r] & X && ``\colim"_n Z_n \ar[r] &  X
} 
\end{equation}
of $\PSh(\Sch_S, \cS)$ corresponding to those in Proposition~\ref{prop:excisionCech}.
We want to show that for every procdh sheaf $F$ the functor $\Map(-, F)$ sends these squares to cartesian squares. 
We consider only $(\ast)$ as the Nisnevich case is heavily studied in the literature (and also, the argument for the left square is the same as for the right but slightly simpler combinatorially). 

By Yoneda it suffices to show that $(\ast)$ becomes a cocartesian after sheafification. The game is the following. Set $A_λ := Z_λ$ for $λ \geq 0$ and $A_{-1} := Y$. Then consider $\cA_0 = \sqcup_{λ \geq -1} A_λ$ (taking the coproduct in $\PSh(\Sch_S, \cS)$), and set $\cA_n = \cA_0^{\times_X (n+1)}$ for $n > 0$. Then $\colim \cA_n \to X$ becomes an isomorphism in $\Shv_{\procdh}(\Sch_S, \cS)$ since it is exactly the sieve associated to the family $\{A_λ \to X\}_{λ \geq -1}$. So it suffices to show that the square $(\ast) \times_X \colim \cA_n$ is cocartesian in $\Shv_{\procdh}(\Sch_S, \cS)$. We now repeatedly use universality of colimits without further comment, \cite[\S 6.1]{HTT}.

To show $(\ast) \times_X \colim \cA_n$ is cocartesian it suffices to show that each $(\ast) \times_X \cA_n$ becomes cocartesian after sheafification, and for this it suffices to show that $(\ast) \times_X (A_{λ_0} \times_X  \dots \times_X A_{λ_n})$ becomes cocartesian after sheafification for each $(λ_0, \dots, λ_n) \in \NN_{\geq -1}^{n+1}$. If any of the $λ_i$ are $\geq 0$, then we have a square of the form
\begin{equation}
\xymatrix{
``\colim"_{μ \geq 0} (Z_{λ_i} \times_X E_μ) \ar[d] \ar[r] & E_{λ_i} \ar[d] \\
``\colim"_{μ \geq 0} (Z_{λ_i} \times_X Z_μ) \ar[r] & Z_{λ_i} 
} 	
\end{equation}
which is cocartesian since $Z_{λ_i} \times_X Z_μ = Z_{λ_i}$ and $ Z_{λ_i} \times_X E_μ = E_{λ_i}$ for all $μ \geq λ_i$. Similarly, if all $λ_i$ are $-1$ but $Y \to X$ is a \emph{closed immersion}, then 
our square becomes 
\[ \xymatrix{
``\colim"_{μ \geq 0} E_μ \ar[d] \ar[r] & Y \ar[d] \\
``\colim"_{μ \geq 0} E_μ \ar[r] & Y.
} \]
Finally, for a general $Y$ and $λ_i = -1$ for all $i$, the square $(\ast) \times_X (A_{λ_0} \times_X  \dots \times_X A_{λ_n})$ is the lower square $(\ast\ast)$ below.
\[ \xymatrix{
``\colim"_{μ \geq 0} \biggl ( E_μ  \biggr ) \ar[r] \ar[d]_{(\id_{E_μ}, \Delta^{n+1}_{E_μ})} \ar@{}[dr]|{(\ast\ast\ast)} & Y \ar[d]^{(\id_Y, \Delta_Y^{n+1})} \\
``\colim"_{μ \geq 0} \biggl ( E_μ \times_{Z_μ} E_μ^{\times_{Z_μ} {(n+1)}} \biggr ) \ar[r] \ar[d]_{pr_2} \ar@{}[dr]|{(\ast\ast)} & Y \times_X Y^{\times_X {(n+1)}} \ar[d]^{pr_2} \\
``\colim"_{μ \geq 0} \biggl ( E_μ^{\times_{Z_μ} {(n+1)}} \biggr ) \ar[r] & Y^{\times_X {(n+1)}}
} \]
Then since $\{E_μ^{\times_{Z_μ} (n+1)} \to Y^{\times_X (n+1)}\}_{μ \in \NN} \sqcup \{Y \stackrel{\Delta_Y^{n+1}}{\to} Y^{\times_X (n+1)}\}$ is a procdh covering family for any $n \geq 0$ and the $Y \to Y^{\times_X (n+1)}$ are closed immersions, the square $(\ast {\ast}\ast)$ and the outside square are both cocartesian by the previous case. Then it follows from the two-out-or-three property for cocartesian squares, \cite[Lem.4.4.2.1]{HTT},  that the third square $(\ast \ast)$ is also cocartesian.
\end{proof}

\section{Homotopy dimension} \label{sec:homotopyDimension}

\subsection{Valuative dimension}

\begin{defi} \label{defi:valuDIm}
Recall that the \emph{valuative dimension} of a scheme is the supremum of the ranks of all valuation rings of residue fields of generic points of $X$ centred on $X$.
\[ \dim_v(X) = \sup \left \{
\dim R\ \middle |\ 
\begin{array}{c}
\exists\ x \in X^{\gen}; R \textrm{ is a valuation ring of } k(x)  \\
\exists\ \Spec(R) \to X \textrm{ compatible with } R \subseteq k(x)
\end{array}
\right \}. \]
\end{defi}

The valuative dimension is of interest to us because it controls the size of $\RZ(X)$. Indeed, the $R$ appearing in the definition of valuative dimension are precisely the reductions of the local rings of the locally ringed topological space $\lim_{Y \in \Mod_X} Y$.

%
%
%
The following basic properties of valuative dimension follow directly from well-known facts about valuation rings.

\begin{lemm}  \label{lemm:valDimProp}
Let $X$ be a scheme.
\begin{enumerate}
 \item If $Y \subseteq X$ is a subscheme, or $f: Y \to X$ is an étale morphism, then  
 \[ \dim_v Y \leq \dim_v X. \]
 \item If $Y \subseteq X$ is a nowhere dense subscheme and $\dim_v(Y)$ is finite, then
 \[ \dim_v Y \lneq \dim_v X. \]
 \item If $Y \to X$ is a proper morphism inducing a bijection $Y^{\gen} = X^{\gen}$ then 
 \[ \dim_v Y = \dim_v X. \]
 \item We have
 \[ \dim_v \AA^1_X = \dim_v X + 1. \]
 \item \label{lemm:valDimProp:ft} If $X$ has finite valuative dimension then so does every $X$-scheme of finite type.
\end{enumerate}
\end{lemm}

\begin{proof}
\begin{enumerate}
 \item[1a.] If $x_0, x_1 \in X$ are points with $x_0 \in \overline{\{x_1\}}$ then there exists a morphism $\Spec(R) \to X$ such that $R \subseteq k(x_1)$ is a valuation ring and the image of $\Spec(R/\m)$ is $x_0$, \cite[00IA]{stacks-project}. Also, if $R' \subseteq k(x_0)$ is another valuation ring, then there exists an extension $R''$ of $R'$ to $R/\m$, \cite[pg.89,Ch.VI,\S 1.3,Cor.3]{Bou64}. Then $R''' = R'' \times_{R/\m} R \subseteq k(x_1)$ is a valuation ring with $\dim R''' = \dim R + \dim R'' \geq \dim R + \dim R'$.

 \item[1b.] If $\Spec(R) \to Y$ is a morphism with $R$ a valuation ring of $k(y)$ for some $y \in Y$ then $R' = R \cap k(f(y))$ is a valuation ring of $k(x)$ equipped with a morphism $\Spec(R') \to X$ such that $R / R'$ is flat, \cite[0539]{stacks-project}. In particular, it satisfies going down, \cite[00HS]{stacks-project}, so $\dim R' \leq \dim R$.

 \item[2.] If $x_0 \neq x_1$ in part 1(a) then $\dim R > 0$.

 \item[3.] Use the valuative criterion for proper morphisms.
 
 \item[4.] Since $\dim_v X = \sup \{ \dim_v W \}$ for integral affine subschemes $W \subseteq X$ we can assume $X$ is integral affine, say $X = \Spec(A)$. Let $K = \Frac(A)$. For any extension $R' / R$ of valuation rings with fraction field extension $K' / K$ and residue extension $κ'/κ$ we have the inequality, \cite[pg.162,Ch.VI,\S 10.3,Cor.2]{Bou64} and \cite[pg.111,Ch.VI,\S 4.4,Prop.5]{Bou64},
\[ \dim R' + \mathrm{tr.deg}(κ'/κ) \leq \dim R + \mathrm{tr.deg}(K'/K). \]
 Applying this to a valuation ring $A[t] \subseteq R' \subseteq \Frac(A)(t)$ with $\dim R' = \dim_v A[t]$ and $R = \Frac(A) \cap R'$ gives
 \[ \dim_v A[t] \leq \dim_v A[t] + \mathrm{tr.deg}(κ'/κ) \leq \dim R + 1 \leq \dim_v A + 1. \]
 
We get the other inequality $\dim_v A[t] \geq \dim_v A + 1$ from the composition valuation ring $A[t] \subseteq K[t]_{(t)} \times_{K} \OO \subseteq K(t)$ where $A \subseteq \OO \subseteq K$ is a valuation ring and $K[t]_{(t)} \to K$ is $t \mapsto 0$. 

 \item[5.] This follows from (4) and (1).
\end{enumerate}
\end{proof}


\subsection{Homotopy dimension}

\begin{rema}
The first author points out that despite the heavy citing of \cite{HTT}, quasi-categories are completely unnecessary for our purposes, and sometimes not even preferable, cf.the proof of Prop.\ref{prop:enoughPoints}, or more specifically, the mess that is Def.\ref{defi:minusOneConnective}. Everything below can be comfortably carried out using appropriate model categories of presheaves of simplicial sets, and the theory of Bousfield localisations as developed in \cite{Hir03}. In fact, many of the quasi-categorical statements that we cite from \cite{HTT} are proved using such model categorical arguments.
\end{rema}




\emph{Truncated spaces.} 
Recall that for $n \geq -2$ one says that a space $K \in \cS$ is \emph{$n$-truncated} if $\Map(D^{n+2}, K) \stackrel{\sim}{\to} \Map(S^{n+1}, K)$, where $D^{n+2}$ is the $(n+2)$-disc, $S^{n+1}$ is its boundary, the $(n+1)$-sphere, so $S^0 = \ast \sqcup \ast$ and $S^{-1} := \varnothing$, \cite[Lem.5.5.6.17]{HTT}. For $n \geq -1$, this is equivalent to asking that $\pi_i(K, k) \cong \ast$ for all $k \in \pi_0(K)$ and all $i > n$, \cite[p.xiv]{HTT}.
This leads to a decreasing sequence of full subcategories
\[ \cS \hookleftarrow \dots \hookleftarrow \cS_{\leq 1} \hookleftarrow \cS_{\leq 0} \hookleftarrow \cS_{\leq -1} \hookleftarrow \cS_{\leq -2} = \{\ast\} \]
which in low degrees is 
\[ \cS_{\leq 1} = \textrm{ 1-groupoids, } \quad 
\cS_{\leq 0} = \textrm{ discrete spaces, } \quad 
\cS_{\leq -1} = \{\varnothing \to \ast\}, \quad
\cS_{\leq -2} = \{\ast\}.
\]
By definition $\cS_{\leq n} \subset \cS$ is the subcategory of $\{S^{n+1} {\to} D^{n+2}\}$-local objects, \cite[Def.5.5.4.1]{HTT}, so localisation%
\footnote{\label{foot:localisationHTT} We mean localisation in the sense of \cite[Def.5.2.7.2]{HTT}. So $\cS \to \cS_{\leq n}$ is the universal functor sending every morphism in the \emph{strong saturation}, \cite[Def.5.5.4.5]{HTT}, of \{$S^{n+1} \to D^{n+2}\}$ to an equivalence, \cite[Prop.5.2.7.12]{HTT}, \cite[Prop.5.5.4.15]{HTT}.%
} %
at $\{ S^{n+1} \to D^{n+2}\}$ is a left adjoint 
\[ (-)_{\leq n}: \cS \to \cS_{\leq n} \]
to inclusion.\footnote{\label{foot:local} Existence of the left adjoint can be deduced from the adjoint functor theorem, \cite[Cor.5.5.2.9]{HTT} or constructed directly using the small object argument applied to the set $\{S^i {\to} D^{i+1}\}_{i > n}$. See Hatcher's textbook \cite[Exam.4.17]{Hatcher} for an extremely concrete construction.}



\emph{Truncated objects.} The notion of $n$-truncatedness is extended to a general ∞-category $T$, such as $T = \PSh(C, \cS)$, by declaring an object $F \in T$ to be \emph{$n$-truncated} if the space $\Map(G, F)$ is $n$-truncated for all objects $G$ of $T$. %
If $T$ admits finite colimits then an object is $n$-truncated if and only if it is $\{\Sigma^{n+1} G \to G\}_{G \in T}$-local 
where $\Sigma^{-1}G := \varnothing$ and $\Sigma^{n+1}G := G \sqcup_{\Sigma^nG} G$. %
Dually, if $T$ admits finite limits, then an object $F$ is $n$-truncated if and only if $F \stackrel{\sim}{\to} \Omega^{n+1}F$  
where $\Omega^{-1}F := \ast$ and $\Omega^{n+1}F := F \times_{\Omega^nF} F$.%

If $T$ is generated under small colimits by a set $C$ of objects of $T$, %
then $F$ is $n$-truncated if and only if it is $\{\Sigma^{n+1} X \to X\}_{X \in C}$-local. %
If furthermore objects in the generating set $C$ are compact then filtered colimits of $n$-truncated objects are $n$-truncated. %
That is, the inclusion $T_{\leq n} \subseteq T$ is $\omega$-accessible. %
Finally, if $T$ is presentable%
\footnote{This means that $T$ admits all small colimits and is of the form $T = \Ind(T')$ for some small category $T'$.} %
then 
the inclusion admits a left adjoint, \cite[Prop.5.5.6.18]{HTT},
\[ (-)_{\leq n}: T \to T_{\leq n}. \]

\emph{Functoriality.} %
Suppose $T, T'$ are ∞-categories admitting finite limits and $ρ: T' \to T$ is a functor. %
As observed above, an object $F$ is $n$-truncated if and only if $F \stackrel{\sim}{\to} \Omega^{n+1}F$. Consequently, if $ρ$ preserves finite limits, then it preserves $n$-truncated objects, cf.\cite[Prop.5.5.6.16]{HTT}, and we obtain the commutative square on the right hand side of the following diagram.
\begin{equation} \label{equa:truncComm}
\xymatrix{
T \ar[d]_{(-)_{\leq n}} \ar[r]^\lambda & T' \ar[d]^{(-)_{\leq n}}
&&
T & \ar[l]_\rho T' \\
T_{\leq n} \ar@{-->}[r]^-{\lambda_{\leq n}} & T'_{\leq n}  
&& 
T_{\leq n} \ar[u]^{inc.} & \ar@{-->}[l]_-{\rho_{\leq n}} T'_{\leq n} \ar[u]_{inc.} 
} 
\end{equation}
If $\rho$ and the inclusions $T_{\leq n} \subseteq T$, $T'_{\leq n} \subseteq T'$ all admit left adjoints $\lambda, (-)_{\leq n}, (-)_{\leq n}$ respectively, then an adjunction argument shows that the functor $\lambda_{\leq n} := (-)_{\leq n} \circ \lambda \circ inc.$
produces the commutative square on the left hand side of Eq.\eqref{equa:truncComm} above. In fact, since the $inc.$ are fully faithful, $\lambda_{\leq n}$ is a left adjoint to $\rho_{\leq n}$. %
Note that $inc.$ and $(-)_{\leq n}$ preserve final objects. So if $λ$ also preserves the final object, we have
\[ 
λ_{\leq n}(\ast) 
= 
(-)_{\leq n} \circ \lambda \circ inc.(\ast)
= 
\ast.
\]
Therefore, for $F \in T$ we have $F_{\leq n} \cong \ast$ implies $λ(F)_{\leq n} \cong \ast$. 

\begin{exam} \label{exam:truncFunc}\ 
\begin{enumerate}
 \item If $\Phi: \Shv_τ(C, \cS) \to \cS$ is any fibre functor\footnote{As in the case of sets, $\Phi$ is a fibre functor if it preserves all colimits and finite limits, cf.\cite[
 Rem.6.3.1.2, 
 Cor.5.5.2.9, 
 Thm.6.1.0.6
]{HTT}.} and $F \in \Shv_τ(C, \cS)$, we have $F_{\leq n} \cong \ast \implies \Phi(F)_{\leq n} \cong \ast$. 

 \item If $X \in C$ is any object and $F \in \Shv(C, \cS)$, 
we have $F_{\leq n} \cong \ast \implies (F|_{X})_{\leq n} \cong \ast$, where
$(-)|_X: \Shv(C, \cS) \to \Shv(C_{/X}, \cS)$ is the restriction functor with $C_{/X}$  equipped with the induced topology: coverings in $C_{/X}$ are precisely those families which are sent to coverings in $C$; the projection $C_{/X} \to C$ is a continuous and cocontinuous morphism of sites.

 \item If $X$ is a qcqs scheme with Noetherian topological space and $F \in \Shv_{\procdh}( \Sch_S )$ we have $F_{\leq n} \cong \ast \implies (\rho_X^*F)_{\leq n} \cong \ast$ in $\Shv(\RZ(X_{\Nis}), \cS)$, cf. Corollary~\ref{coro:RZColimitLim}.
\end{enumerate}
\end{exam}

As a right adjoint, global sections $\Map(\ast, -)$ does not preserve connectivity in general. Homotopy dimension describes how badly this fails.


\begin{defi}[{\cite[Prop.6.5.1.12, Def.7.2.1.1]{HTT}}] \label{defi:homDim}
One says the $\infty$-topos $\Shv_\tau(C, \cS)$ has \emph{homotopy dimension $\leq d$} if for every sheaf $F \in \Shv_\tau(C, \cS)$ we have 
\[ F_{\leq d-1} \cong \ast \textrm{ implies }\Map(*, F)_{\leq -1} \cong \ast. \]
Note that the latter condition is equivalent to $\Map(\ast, F)_{\leq -1} \neq \varnothing$.
\end{defi}

\begin{rema} \label{rema:ndconn}
More generally, if $\Shv_\tau(C, \cS)$ has homotopy dimension $\leq d$ then we have 
\[ F_{\leq d+n} \cong \ast \textrm{ implies } \Map(*, F)_{\leq n} \cong \ast \]
for all $n \geq -1$, \cite[Def.7.2.1.6, Lem.7.2.1.7]{HTT}.
\end{rema}

\begin{rema} \label{rema:locFHD}
Even more generally, one says that $\Shv_τ(C, \cS)$ is \emph{locally of finite homotopy dimension} if for every $X \in C$ there is some $d_X < ∞$ such that $F_{\leq d_X + n} \cong \ast$ implies $\Map(X, F)_{\leq n} \cong \ast$ for all $n \geq -1$. 

By the canonical adjunction $\Shv_τ(C, \cS) \rightleftarrows \Shv_τ(C_{/X}, \cS)$ (cf. Example \ref{exam:truncFunc}(2)), this is equivalent to asking that each $\Shv_τ(C_{/X}, \cS)$ has finite homotopy dimension.
\end{rema}


\begin{exam} \label{exam:Rlim}
Consider the category $\NN = \{0 \to 1 \to 2 \to \dots \}$. An object of $\PSh(\NN, \cS)$ is a diagram $\dots \to K(2) \to K(1) \to K(0)$ and the global sections functor is $\{K(n)\}_{n\in \NN} \mapsto \lim_{n\in \NN} K(n)$. If $(K(n))_{\leq 0} \cong \ast$, that is, $π_0K(n) \cong \ast$ for all $n$, then it follows from the short exact sequences of pointed sets, \cite[\S 7.4]{BK72},
\begin{equation}
\ast \to {\lim_{n \in \NN}}^1 π_1 K(n) \to π_0 \lim_{n \in \NN} K(n) \to \lim_{n \in \NN} π_0 K(n) \to \ast
\end{equation}
that $(\lim K(n))_{\leq -1} \cong \ast$. So the topos $\PSh(\NN, \cS)$ has homotopy dimension $\leq 1$. The sequence of non-empty discrete spaces $K(n) = \NN_{\geq n}$ shows that the homotopy dimension is $\not \leq 0$ since $K(n)_{\leq {-1}} = \ast$ for all $n$ but $\hom(\ast, \{K(n)\}_{n\in \NN})_{\leq -1} = \varnothing$. \\
%
\end{exam}

\begin{exam}[{\cite[Cor.3.11, Thm.3.18]{CM21}}] \label{exam:RZhomDim}
If $C_λ$ is a filtered system of finitary\footnote{A site is finitary if it has finite limits and every covering family is refineable by a finite one.} excisive\footnote{A site is excisive if for all $U \in C$, the functor $F \mapsto \Map(U, F)$ commutes with filtered colimits.} sites with colimit $C$, then Clausen and Mathew show that $\Shv(C, \cS)$ has homotopy dimension $\leq d$ if all $\Shv(C_λ, \cS)$ do. Using this they show that for any qcqs algebraic space whose underlying topological space has Krull dimension $\leq d$, the ∞-topos $\Shv(X_{\Nis}, \cS)$ has homotopy dimension $\leq d$.

It also follows from this that if $X$ is a qcqs scheme of valuative dimension $d$ then $\RZ(X_{\Nis})$ has homotopy dimension $\leq d$.
\end{exam}

\begin{theo} \label{theo:bounded}
Let $S$ be a qcqs scheme of finite valuative dimension $d \geq 0$ with Noetherian underlying topological space. 
Then $\Shv_{\procdh}(\Sch_S, \cS)$ has homotopy dimension $\leq 2d$.
\end{theo}

\begin{proof}
The proof is by induction on the valuative dimension of $S$.
Suppose $F \in \Shv_{\procdh}(\Sch_S, \cS)$ has $F_{\leq 2d-1} = \ast$. We want to show that $F(S)_{\leq -1} \cong \ast$.

Since $S$ has Noetherian topological space, 
\[\rho^*=\rho_S^*: \Shv_{\procdh}(\Sch_S, \cS) \to \Shv(\RZ(S_{\Nis}), \cS)\]
is a left adjoint of a morphism of ∞-topoi, Cor.\ref{coro:RZColimitLim}, so $(\rho^*F)_{\leq 2d-1} \cong \ast$, Exam.\ref{exam:truncFunc}. Since the homotopy dimension of $\Shv(\RZ(S_{\Nis}), \cS)$ is $\leq d$, Exam.\ref{exam:RZhomDim}, we have $(\rho^*F)(S)_{\leq -1} \cong \ast$, that is, the space $(\rho^*F)(S)$ is non-empty. Concretely, one can calculate $(\rho^*F)(S)$ as $\colim_{Y \in \Mod_S} F(Y)$, Remark \ref{rema;rhoX}, 
so 
we can find a modification $Y \to S$ such that $F(Y)$ is non-empty.  

If $d = 0$, we have $Y = S$ and we are done with this step. 

If $d > 0$, up to refining $Y$ we can assume that $Y^{\gen} = S^{\gen}$, Lem.\ref{lemm:rema:genXY}. In particular, there exists a nowhere dense non-empty closed subscheme of finite presentation $Z_0 \subseteq S$ such that 
$Y \to S$ is an isomorphism over $S\backslash Z_0$ and
$E_0 := Z_0 \times_S Y$ is also a nowhere dense closed subscheme of finite presentation. Note we now have $0 \leq \dim_v Z_0 \leq d-1$ and similar for $E_0$, Lem.\ref{lemm:valDimProp}. We continue to have $(F|_{\Sch_{Z_n}})_{\leq 2d-1} \cong \ast$ and $(F|_{\Sch_{E_n}})_{\leq 2d-1} \cong \ast$, Exam.\ref{exam:truncFunc}. By the induction hypothesis $\Shv_{\procdh}(\Sch_{Z_n}, \cS)$ and $\Shv_{\procdh}(\Sch_{E_n}, \cS)$ all have homotopy dimension $\leq 2d-2$, so $F(Z_n)_{\leq 1}, F(E_n)_{\leq 1} \cong \ast$ for all $n$, Rem.\ref{rema:ndconn}. We have seen that $\PSh(\NN, \cS)$ has homotopy dimension $\leq 1$, Exam.\ref{exam:Rlim} so $(\lim_{n \in \NN} F(Z_n))_{\leq 0} \cong \ast$ and $(\lim_{n \in \NN} F(E_n))_{\leq 0} \cong \ast$. Combining this with $F(Y)_{\leq -1} \cong \ast$, cartesianness of the square, Prop.\ref{prop:hyperExci},
\[ \xymatrix{
F(S) \ar[r] \ar[d] & F(Y) \ar[d] \\
\lim_{n \in \NN} F(Z_n) \ar[r] & \lim_{n \in \NN} F(E_n)
} \]
implies that $F(S)_{\leq -1} \cong \ast$. Indeed, both $\lim_{n \in \NN} F(Z_n)$ and $\lim_{n \in \NN} F(E_n)$ are non-empty connected and $F(Y)$ is non-empty, so the pullback is also non-empty.
\end{proof}

\begin{coro} \label{coro:finCohDim}
Let $S$ be a qcqs scheme of finite valuative dimension $d \geq 0$ with Noetherian underlying topological space. 
For any sheaf of abelian groups $F \in \Shv_{\procdh}(\Sch_S, \Ab)$ we have
\[ H_{\procdh}^n(S, F) = 0; \qquad n > 2d. \]
\end{coro}

\begin{proof}
This is \cite[Cor.7.2.2.30]{HTT}. Cf. also \cite[Def.7.2.2.14, Rem.7.2.2.17]{HTT}.
\end{proof}

\subsection{Hypercompleteness}

\begin{defi}
An ∞-topos is called \emph{hypercomplete} if for every object $F$ we have
\[ F = \lim_n τ_{\leq n}F. \]
\end{defi}

\begin{exam} \label{exam:hypercomplete} \ 
\begin{enumerate} 
 \item The category $\cS$ of spaces is hypercomplete, \cite[IX.3.1]{BK72}. 

 \item (Jardine) If $\Shv(C, \cS)$ is locally of finite homotopy dimension, Rem.\ref{rema:locFHD}, then $\Shv(C, \cS)$ is hypercomplete, \cite[Cor.7.2.1.12]{HTT}.
 \end{enumerate}
\end{exam}

\begin{coro} \label{coro:procdhHypercomplete}
Let $S$ be a qcqs scheme of finite valuative dimension with Noetherian underlying topological space. Then $\Shv_{\procdh}(\Sch_S, \cS)$ is hypercomplete.
\end{coro}

\begin{proof}
Lemma~\ref{lemm:valDimProp}\eqref{lemm:valDimProp:ft} %
and 
Lemma~\ref{lemm:noethTopFt} %
say that each $X \in \Sch_S$ also have finite valuative dimension with Noetherian underlying topological space so by Theorem~\ref{theo:bounded} the ∞-topos $\Shv_{\procdh}(\Sch_S)$ is locally of finite homotopy dimension, so by Example~\ref{exam:hypercomplete} it is hypercomplete.
\end{proof}

\subsection{Points of ∞-topoi}

As in the 1-topos case, one says that an ∞-topos $\Shv(C, \cS)$ has \emph{enough points} when a morphism $f$ is an equivalence if and only $\phi^*f$ is an equivalence for every geometric morphism of ∞-topoi $\phi^*: \Shv(C, \cS) \rightleftarrows \cS: \phi_*$, \cite[Rem.6.5.4.7]{HTT}. Unlike the 1-topos case, in general, detecting equivalences is not equivalent to detecting coverings unless $\Shv(C, \cS)$ is hypercomplete, \cite[Prop.A.4.2.1]{SAG}.%
\footnote{This is unsurprising as the definition of sieve in \cite{HTT} basically encodes the notion of \v{C}ech descent, cf.Eq.\eqref{equa:cech2}.
}

The purpose of this subsection is to prove Proposition~\ref{prop:enoughPoints} which says that if the 1-topos of a site has enough points and the ∞-topos is hypercomplete, then the ∞-topos also has enough points. 

\begin{rema}
	This should be compared to Lurie's result  \cite[Thm.A.4.0.5]{SAG}, which says that any ∞-topos which is locally coherent and hypercomplete has enough points. We cannot apply this to the procdh topos because it is not locally coherent.%
	\footnote{For example the generator covering family Eq.\eqref{equa:twodash} in Remark~\ref{rema:genprocdh} does not contain a finite subfamily which is still covering. For any choice of finite subcovering there will be an Artinian quotient $\Spec \ZZ[x,y] / \langle x^n, y^n \rangle \to \Spec \ZZ[x, y]$ which does not factor through this subfamily.}

Conversely, combining Deligne's 1-categorical completeness theorem with our Proposition~\ref{prop:enoughPoints} recovers \cite[Thm.A.4.0.5]{SAG}.
\end{rema}

We apply Proposition~\ref{prop:enoughPoints} to the procdh topos in Corollary~\ref{coro:enoughInfinityPoints}.

\begin{lemm} \label{lemm:truncatedPoints}
Let $(C, τ)$ be a site admitting finite limits. 
Every fibre functor $\phi^*_{\Set}: \Shv(C, \Set) \to \Set$ of the 1-topos $\Shv(C, \Set)$ induces a fibre functor $\phi^*_{\cS}$ of the ∞-topos $\Shv(C, \cS)$ fitting into a commutative square as below. 

Conversely, 
every fibre functor $\phi^*_{\cS}: \Shv(C, \cS) \to \cS$ of the ∞-topos $\Shv(C, \cS)$ induces a fibre functor $\phi^*_{\Set}$ of the 1-topos $\Shv(C, \Set)$ fitting into a commutative square as below.
\[ \xymatrix{
\Shv(C, \cS) \ar[r]^-{(-)_{\leq 0}} \ar[d]_{\phi^*_{\cS}} & \Shv(C, \Set) \ar[d]^{\phi^*_{\Set}} \\
\cS \ar[r]^{(-)_{\leq 0}} & \Set
} \]
\end{lemm}

\begin{proof}
Since $C$ admits finite limits, in both the 1-category and ∞-category setting there is a canonical correspondence between fibre functors $\phi^*_?$ of $\PSh(C, ?)$ and pro-objects $P$ of $C$,  given by $(P:\Lambda \to C) \mapsto (F \mapsto \colim_{\lambda\in \Lambda} F(P_λ))$ and $\phi^* \mapsto (P:\int_C F \to C)$, \cite[Cor.5.3.5.4]{HTT}. So, for $\phi^*_{\Set}: \Shv(C, \Set) \to \Set$ as in the lemma, we automatically get the extended diagram on the left.
\begin{equation} \label{equa:PShShvSet}
\xymatrix@!R=3pt{
\PSh(C, \cS) \ar[r]^-{(-)_{\leq 0}} \ar[d] \ar@/_3em/[dd]_{\phi^*_{\PSh}} & \PSh(C, \Set) \ar[d] 
&&
\PSh(C, \cS) \ar[r]^-{(-)_{\leq 0}} \ar[d] & \PSh(C, \Set) \ar[d] \ar@/^3em/[dd]^{\phi^*_{\PSh}} \\
\Shv(C, \cS) \ar[r]^-{(-)_{\leq 0}} \ar@{}[d] & \Shv(C, \Set) \ar[d]^{\phi^*_{\Set}} 
&&
\Shv(C, \cS) \ar[r]^-{(-)_{\leq 0}} \ar@{}[d] \ar[d]_{\phi^*_{\cS}} & \Shv(C, \Set)  \\
\cS \ar[r]^{(-)_{\leq 0}} & \Set
&&
\cS \ar[r]^{(-)_{\leq 0}} & \Set
} 
\end{equation}
Since $\Shv(C, \cS)$ is the topological localisation of $\PSh(C, \cS)$ generated by covering sieves, \cite[Def.5.5.4.5, Def.6.2.1.4]{HTT} 
and $\phi^*_{\PSh}$ preserves colimits, to show that $\phi^*_{\PSh}$
factors through 
$\Shv(C, \cS)$, it suffices to show it sends covering sieves to equivalences. 
Any sieve\footnote{Or more generally, any monomorphism, i.e., any morphism $F \to G$ such that $F \to F \times_G F$ is an equivalence.} in $\PSh(C, \cS)$ is sent by $\phi^*_{\PSh}$ to an inclusion of connected components $A \to A \sqcup B$ in $\cS$. Such an inclusion is an equivalence if and only if it induces an isomorphism on connected components, i.e., if it is an isomorphism after $τ_{\leq 0}$. So the result follows from the fact that covering sieves are sent to isomorphisms in $\Set$ by the outside square in Eq.\eqref{equa:PShShvSet}.

Conversely, by the same argument, any $\phi_\cS^*$ gives rise to a $\phi^*_{\PSh}$ as in the diagram on the right. Again, this factors through $\Shv(C, \Set)$ if and only if it sends covering sieves to isomorphisms, but this is guaranteed by the fact that $\PSh(C, \cS) \to \Shv(C, \cS)$ sends covering sieves to equivalences.
\end{proof}

\begin{defi} \label{defi:minusOneConnective}
A morphism $f: E \to B$ in an ∞-topos $\Shv_τ(C, \cS)$ is \emph{$(-1)$-connective} if it satisfies the following equivalent conditions:
\begin{enumerate}
 \item $f$ is an effective epimorphism, i.e., $\colim_n E^{\times_B n} \to B$ is an equivalence, \cite[pg.584]{HTT}.

 \item $f_{\leq 0}$ is an effective epimorphism in the 1-topos $\Shv_τ(C, \Set)$. In other words, 
 $\textrm{coeq}(E_{\leq 0} \times_{B_{\leq 0}} E_{\leq 0} \rightrightarrows E_{\leq 0})_{\leq 0} \to B_{\leq 0}$ is an isomorphism, \cite[Prop.7.2.1.14]{HTT}.
\end{enumerate}
For $n \geq 0$, a morphism $f: E \to B$ is \emph{$n$-connective} if it satisfies the following equivalent conditions:
\begin{enumerate}
 \item[(a)] $f$ is $(-1)$-connective and $π_k(f) = \ast$ for $0 \leq k < n$,\footnote{Here $π_k(f) = τ_{\leq 0}^{/E}(\Omega^nE \to E)$ where $\Omega^0E = E \times_B E$ and $\Omega^nE = E \times_{\Omega^{n-1}E} E$.} \cite[Def.6.5.1.1, pg.657, Def.6.5.1.10]{HTT}.

 \item[(b)] $f$ is $(-1)$-connective and $E \to E \times_B E$ is $(n{-}1)$-connective, \cite[Prop.6.5.1.18]{HTT}.

 \item[(c)] $τ_{\leq n-1 }^{/B} f \to B$ is an equivalence, 
 where $τ_{\leq m}^{/B}: \Shv_τ(C, \cS)_{/B} \to (\Shv_τ(C, \cS)_{/B})_{\leq m}$ is the truncation in the slice ∞-topos
  \cite[Prop.6.5.1.12]{HTT}.\footnote{Notice that the other two conditions are independent of whether we consider $E \to B$ as an object of $\Shv_τ(C, \cS)$ or $\Shv_τ(C, \cS)_{/B}$, so \cite[Prop.6.5.1.12]{HTT} really does apply.}
\end{enumerate}
\end{defi}

\begin{lemm} \label{lemm:detectConnectivity}
Let $C$ be a site and suppose that the 1-topos $\Shv(C, \Set)$ has enough points as a 1-topos. Then for any $n \geq -1$ and any morphism $f \in \Shv(C, \cS)$, the morphism $f$ is $n$-connective if 
for every fibre functor $\phi^*: \Shv(C, \cS) \to \cS$ the morphism of spaces $\phi^*f$ is $n$-connective.
\end{lemm}

\begin{proof}
The case $n = -1$ follows from Lemma~\ref{lemm:truncatedPoints}  and the fact that in a 1-topos with enough points, effective epimorphisms are detected by points.\footnote{Indeed, $Y \to X$ is an effective epimorphism in a 1-topos if and only if $\textrm{coeq}(Y \times_X Y \rightrightarrows Y) \to X$ is an isomorphism.}


The $n > -1$ cases then follows immediately, since by definition a morphism $E \to B$ is $n$-connective if and only if each of the $n+1$ iterated diagonals $E \to B$, $E \to E \times_B E$, $E \to E \times_{E \times_B E} E$, etc, is an effective epimorphism.
\end{proof}

Lemma~\ref{lemm:detectConnectivity} can be converted to a statement about truncations using the following lemma.

\begin{lemm} \label{lemm:connTruncEquiv}
Suppose $\Shv_τ(C, \cS)$ is an ∞-topos and $f:E \to B$ a morphism. We have the following implications.
\begin{enumerate}
 \item $f$ is $n$-connective $\implies$ $f_{\leq n-1}$ is an equivalence.
 \item $f_{\leq n}$ is an equivalence $\implies$ $f$ is $n$-connective.
\end{enumerate}
\end{lemm}

\begin{rema}
Clearly, the above implications are strict. The morphism 
$f: \ast \to S^{n}$ is not $n$-connective but $f_{\leq n-1}$ is an equivalence%
, and the morphism 
$f: S^{n} \to \ast$ is $n$-connective but $f_{\leq n}$ is not an equivalence.
\end{rema}

\begin{proof}
Recall that a morphism $f$ is $n$-connective if and only if $τ_{\leq n-1}^{/B} f \to B$ is an equivalence in the slice category $\Shv(C, \cS)_{/B}$, Def.\ref{defi:minusOneConnective}. Applying the commutative square Eq.\eqref{equa:truncComm} to the adjunction 
\[\textrm{forget}: \Shv(C, \cS)_{/B} \rightleftarrows \Shv(C, \cS): - \times B\]
we get a commutative square 
\[ \xymatrix@C=3em{
\Shv(C, \cS)_{/B} \ar[r]^-{\textrm{forget}} \ar[d]^{τ_{\leq n-1}^{/B}} & \Shv(C, \cS) \ar[d]^{(-)_{\leq n-1}} \\
\Shv(C, \cS)_{/B} \ar[r]^-{\textrm{forget}_{\leq n}} & \Shv(C, \cS)_{\leq n}
} \]
Hence, we find that if $f$ is $n$-connective then $E_{\leq n-1} \to 
B_{\leq n-1}$ is an equivalence. 

Conversely, suppose that $E_{\leq n} \to B_{\leq n}$ is an equivalence. We will show Def.\ref{defi:minusOneConnective}(a) is satisfied. Since $E_{\leq n} \to B_{\leq n}$ is an equivalence, $E_{\leq 0} \to B_{\leq 0}$ is an equivalence, so $f$ is certainly $(-1)$-connective. Furthermore, we have $π_k(E) \cong f^*π_k(B)$ for all $k \leq n$, \cite[Lem.6.5.1.9]{HTT}. Then the long exact sequence 
\[ \cdots \to f^*\pi_{k+1}(B) \to \pi_k(f) \to \pi_k(E) \to f^*\pi_k(B)\to \cdots\]
from \cite[Rem.6.5.1.5]{HTT} applied to $E \to B \to \ast$ shows that $π_k(f) \cong \ast$ for all $k < n$. 
\end{proof}

\begin{prop} \label{prop:enoughPoints}
Let $(C, τ)$ be a small site admitting finite limits such that the 1-topos $\Shv(C, \Set)$ has enough points as a 1-topos. Then the hypercompletion $\Shv(C, \cS)^\wedge$ has enough points as an ∞-topos.
\end{prop}

\begin{proof}
Suppose $f$ is a morphism of $\Shv(C, \cS)$ such that $\phi^*f$ is an equivalence for all fibre functors $\phi^*: \Shv(C, \cS)^\wedge \to \cS$. Then 
$(\phi^*f)_{\leq n}$ is an equivalence for all $n$ and $\phi^*$, 
so $\phi^*f$ is $n$-connective for all $n$ and $\phi^*$, Lem.\ref{lemm:connTruncEquiv}, 
so $f$ is $n$-connective, Lem.\ref{lemm:detectConnectivity}, 
so $f_{\leq n - 1}$ is an equivalence for all $n$, Lem.\ref{lemm:connTruncEquiv}, 
so $f = \lim_n f_{\leq n}$ is an equivalence in $\Shv(C, \cS)^\wedge$ by hypercompleteness.
\end{proof}

The proof of Proposition~\ref{prop:enoughPoints} used a lot of ∞-categorical machinary compressed into Definition~\ref{defi:minusOneConnective}, so as a sanity check, we also give a proof using model categories.

\begin{proof}[Alternative proof of Proposition~\ref{prop:enoughPoints}.]
We have the following facts.
\begin{enumerate}
 
  \item $\PSh(C, \cS)[S^{-1}] \cong (L_S\PSh(C, \Set_{\Delta}))^\circ$ for any set of morphisms $S$ in $\PSh(C, \Set_{\Delta})$. 
  
  That is, the localisation (in the sense of \cite[Def.5.2.7.2]{HTT}) of the ∞-category $\PSh(C, \cS)$ at the image of $S$ is the ∞-category associated to the Bousfield localisation of the simplicial model category $\PSh(C, \Set_{\Delta})$ (with either the injective or projective model structures). 

  \item $\Shv(C, \cS)$ is the localisation of $\PSh(C, \cS)$ at the class of \v{C}ech hypercoverings.
  
  \item $\Shv(C, \cS)^\wedge$ is the localisation of $\PSh(C, \cS)$ at the class $HR$ of all hypercoverings, \cite[Thm.6.5.3.12]{HTT}.
 
  \item the weak equivalences in $L_{HR}\PSh(C, \Set_{\Delta})$ are the morphisms such that each of the induced maps
\begin{align*}
π_0E &\to π_0B \\
π_n(E|_{C_{/X}}, b) &\to π_n(B|_{C_{/X}}, b), \qquad X \in C, b \in B(X)
\end{align*}
are isomorphisms of sheaves, \cite[Prop.2.8]{Jar87}, \cite{DHI04}.
\end{enumerate}
Putting these facts together, the result holds almost by definition.
\end{proof}

\begin{coro} \label{coro:enoughInfinityPoints}
Let $S$ be a qcqs scheme of finite valuative dimension $d \geq 0$ with Noetherian underlying topological space. Then $\Shv_{\procdh}(\Sch_S, \cS)$ has enough points as an ∞-topos.
\end{coro}

\begin{proof}
The 1-topos $\Shv_{\procdh}(\Sch_S, \Set)$ has enough points, Thm.\ref{theo:procdhEnoughPoints}, and the ∞-topos $\Shv_{\procdh}(\Sch_S, \cS)$ is hypercomplete, Cor.\ref{coro:procdhHypercomplete}, so it has enough points by Proposition~\ref{prop:enoughPoints}.
\end{proof}

Using points we get a quick proof of the fact that hyperdescent implies excision (which we already knew by Proposition~\ref{prop:hyperExci}).

\begin{prop} \label{prop:hyperExciPoint}
Suppose that $S$ is a qcqs scheme of finite valuative dimension with Noetherian topological space, and $F \in \PSh(\Sch_S, \cS)$ is a presheaf of spaces. If $F$ satisfies procdh hyperdescent then $F$ satisfies excision, cf.Thm.\ref{theo:decsentStateIntro}.
\end{prop}

\begin{proof}
By Yoneda, it suffices to show that the squares \eqref{equa:NisPorcdhSq}
corresponding to those in Proposition~\ref{prop:excisionCech} are sent to cocartesian squares in the hypercomplete ∞-topos $\Shv_{\procdh}(\Sch_S, \cS)$.
Here $W = U \times_X V$ and $E_n = Z_n \times_X Y$. Since $\Shv_{\procdh}(\Sch_S, \cS)$ has enough points, Cor.\ref{coro:enoughInfinityPoints}, it suffices to show that for every procdh local ring $R = \OO \times_K A$, evaluation of the corresponding representable presheaves on $\Spec(R)$ gives cocartesian squares of spaces. Consider the left square. The horizontal morphisms are monomorphisms, so it suffices to show that $V(R) \setminus W(R) \to X(R) \setminus U(R)$ is an isomorphism. This follows directly from $R$ being henselian, \cite[04GG(8)]{stacks-project}, and the condition that $V \to X$ is an isomorphism outside of $U$. The square on the right uses the same argument with the valuative criterion for properness for the equivalence $Y(R) \setminus \colim_n E_n(R) \stackrel{\sim}{\to} X(R) \setminus \colim_n Z_n(R)$.
\end{proof}

\subsection{Counterexample} \label{subsec:counterHomDim}

We give an example that suggests the bound on the homotopy dimension in Theorem~\ref{theo:bounded} is strict.

\begin{prop} \label{prop:Nisprocdh}
Let $S$ be scheme with finitely many generic points and such that for all $s \in S$ the henselisation $\OO_{S,s}^h$ is procdh local (e.g., a smooth curve over a field, or a procdh local scheme). Suppose $F \in \Shv_{\Nis}(\Sch_S, *)$ is a Nisnevich sheaf of sets / abelian groups / chain complexes of abelian groups / spaces. 
Then $F(S) = F_{\procdh}(S)$.
\end{prop}
%
%

\begin{proof}
Since $F$ and $F_{\procdh}$ are both Nisnevich sheaves, and the (small) Nisnevich ∞-topos is hypercomplete, it suffices to show that
%
%
%
%
for $s \in S$ we have $\phi^*_{\Nis,s}F = \phi^*_{\Nis,s}F_{\procdh}$
where $\phi^*_{\Nis,s}G$ %
means $\colim_{s \to V \to S} G(V)$ and the colimit is over all factorisations with $V \to S$ étale of finite presentation. By assumption $\OO_{S,s}^h$ is procdh local.
Note that $\{s \to V \to S : V \in S_{\Nis} \}$ is a pro-object of $\Sch_S$ and $\Spec(\OO_{S,s}^h) = \lim_{\substack{s \to V \to S \\ V \in S_{\Nis}}} V$ 
so
the functor $\phi^*_{\Nis,s}$ is the fibre functor of $\Shv_{\procdh}(\Sch_S)$ 
associated to the procdh local ring $\OO_{S,s}^h$. Hence, $\phi^*_{\Nis,s}F = \phi^*_{\Nis,s}F_{\procdh}$. 
\end{proof}

\begin{exam} \label{exam:counterCohDim}
The idea is to find a scheme $S$ of valuative dimension $1$, a proabstract blowup square
\begin{equation} \label{equa:examSq}
\xymatrix{
``\underset{n \in \NN}{\colim}" E_n \ar[r] \ar[d] & X \ar[d] \\
``\underset{n \in \NN}{\colim}" Z_n \ar[r] & S 
}
\end{equation}
and a sheaf of abelian groups $F$ such that 
\begin{align} 
H_{\procdh}^i(X, F) = 0 & \textrm{ for all } i \label{equa:HXF} \\
H_{\procdh}^i(Z_n, F) = 0 & \textrm{ for all } i \label{equa:HZn} \\
H_{\procdh}^i(E_n, F) = 0 & \quad i > 0. \label{equa:HEn}
\end{align}
Then the corresponding long exact sequence is
\[ \dots \to 0 \to H_{\procdh}^i(S, F) \to R^{i-1}\lim_{n \in \NN} F(E_n) \to 0 \to \dots. \]
So if we can also arrange that the ${\lim}^1_{n \in \NN}$ is non-zero, we have non-zero $H_{\procdh}^2(S, F)$ giving a counterexample to the statement that $\Shv_{\procdh}(\Sch_S, \cS)$ has homotopy dimension $\leq 1$.

Our example is two affine lines joined at the origin, $S = \AA^1 \sqcup_{\{0\}} \AA^1$, with normalisation $X = \AA^1_k \sqcup \AA^1_k$ and associated 
proabstract blowup square
\begin{equation} \label{equa:examSq}
\xymatrix{
``\underset{n \in \NN}{\colim}" E_n \ar[r] \ar[d] & X \ar[d] \\
``\underset{n \in \NN}{\colim}" Z_n \ar[r] & S 
}
\end{equation}
where $Z_n$ is the $n$th thickening of the origin and $E_n = Z_n \times_S X$. Note that $E_n$ decomposes into two components corresponding to the components of $X$, namely, $E_n = \Spec k[x] / x^n \sqcup \Spec k[y] / y^n =: E_{n,x} \sqcup E_{n,y}$.

To a sequence of abelian groups $\dots A_2 \to A_1 \to A_0$ we associated the presheaf of abelian groups on $\Sch_S$,
\[ F: T \mapsto \left \{ \begin{array}{lr}
A_m, & m = \min\{ i\ |\ \exists\ T \to E_{i,x} \to S\} \\
0, & \forall i, \ \not \exists\  T \to E_{i,x} \to S.
\end{array} \right . \]
For $τ = \Nis, \procdh$, let $F_{τ}$ be the $τ$-sheafifiaction of $F$ considered as a presheaf of chain complexes. So $H^i(F_τ(S)) = H_τ^i(S, F)$, etc. 
The only $V \in X_{\Nis}$ admitting an $S$-morphism $V \to E_{i,x}$ is the empty scheme, so $F_{\Nis}(X) = 0$ and by Proposition~\ref{prop:Nisprocdh} we have $F_{\procdh}(X) = F_{\Nis}(X)$, so \eqref{equa:HXF} holds. We have \eqref{equa:HZn} for the same reason. The condition \eqref{equa:HEn} holds because $E_n$ is a disjoint union of procdh local schemes. More precisely we have 
$$ F_{\procdh}(E_n) 
= F_{\procdh}(E_{n,x}) \times F_{\procdh}(E_{n,y}) 
$$
$$
= F(E_{n,x}) \times F(E_{n,y}) 
= A_n \times \{0\}.$$
where the $0$ is because there are no factorisations $E_{n,y} \to E_{i,x}$ for $n > 1$. So choosing any sequence $(A_n)_{n \in \NN}$ with ${\lim}^1_{n \in \NN} A_n \neq 0$ produces an $F$ with $H^2_{\procdh}(S, F) \neq 0$.
\end{exam}

\begin{rema}
It seems possible to push the above technique further to at least show that the cohomological dimension of certain surfaces is 4. For example, Gabber proposed trying the surface $S \times S$, where $S$ is the scheme from Example~\ref{exam:counterCohDim}.
%
\end{rema}

\section{Application to $K$-theory}  

In this section, we work with the category $\Sch^\qcqs$ of qcqs schemes and 
the full subcategory $\Sch^\noe$ of noetherian schemes.
For $\cC = \Spt, D(\ZZ), D(\QQ), D(\FF_p)$, 
we write $\PSh(\Sch^\qcqs,\cC)$ for the $\infty$-category of presheaves on $\Sch^\qcqs$ with values in $\cC$ and write 
$\Shv_{\procdh}(\Sch^\qcqs,\cC) $ for the full subcategory of the procdh sheaves 
with 
the sheafification functor
\[ a_\procdh: \PSh(\Sch^\qcqs,\cC) \to \Shv_{\procdh}(\Sch^\qcqs,\cC). \]

\begin{rema}
Cf. \cite[Rem.4.1.2]{BS13}. Sheafification on large sites is rare in the literature---and for good reason, it can fail to exist. Indeed, the fpqc topology on $\Sch^{\qcqs}$ does not admit a sheafification functor, essentially because fpqc coverings can have arbitrarily large cardinality, see Example~\ref{exam:fpqc}. For the procdh topology (and many other topologies), one way to proceed is as follows. Choose an uncountable strong limit cardinal $κ$, and implicitly set $\Sch^\qcqs$ to be the category $\Sch^{\qcqs, κ}$ of qcqs schemes of cardinality $< κ$. Since procdh coverings are refinable by covering families whose morphisms are of finite presentation, for any second larger uncountable strong limit cardinal $κ < κ'$ the following two obvious squares commute.
\[ \xymatrix{
\PSh(\Sch^{κ',\qcqs},\cC) 
\ar[d]_{\textrm{restriction}}
\ar@/^0.5ex/[r]^-{a_{\procdh}}  &  
\ar@/^0.5ex/[l]^-{\textrm{inclusion}}
\Shv_{\procdh}(\Sch^{κ',\qcqs},\cC) \ar[d]^{\textrm{restriction}} \\
\PSh(\Sch^{κ,\qcqs},\cC) 
\ar@/^0.5ex/[r]^-{a_{\procdh}}  &  
\ar@/^0.5ex/[l]^-{\textrm{inclusion}}
\Shv_{\procdh}(\Sch^{κ,\qcqs} ,\cC)
} \]
\end{rema}

\begin{exam} \label{exam:fpqc}
Let ${\operatorname {Field}}_\QQ$ be the category of all field extensions of $\QQ$ and consider the presheaf $F$ that sends an extension $K/\QQ$ to the the smallest ordinal $ω$ of the same cardinality as $K$, so  $|F(K)| = |K|$. Note that for any inclusion $K \subseteq L$ with $|K| = |L|$ we have \emph{equality} $F(K) = F(L)$. Suppose that $F$ has a sheafification $F_{\operatorname{fpqc}}$ for the topology on ${\operatorname {Field}}_\QQ$ induced by the fpqc topology, and set $κ = |F_{\operatorname{fpqc}}(\QQ)|$. Choose an extension $K/\QQ$ with $κ < |K|$. 
Let $R = im(\hom(K, -) \to \hom(\QQ, -))$ be the sieve generated by $K$ so $R(L)$ is $\ast$ or $\varnothing$ according to whether there exists a morphism $K \to L$ or not.
All transitions $F(K) \to F(K')$ are inclusions, so $F \to F_{\operatorname{fpqc}}$ is an inclusion of presheaves. This implies that $\hom(R, F) \to \hom(R, F_{\operatorname{fpqc}})$ is also injective. But we have $\hom(R, F) = \lim_{L \in R} F(L) = F(K) = |K|$, and therefore $κ < |K| = |\hom(R, F)| \leq |\hom(R, F_{\operatorname{fpqc}})|$. Since $κ = F_{\operatorname{fpqc}}(\QQ) = 
\hom(\hom(\QQ,-), F_{\operatorname{fpqc}})$, we cannot have 
$\hom(R, F_{\operatorname{fpqc}}) = \hom(\hom(\QQ,-), F_{\operatorname{fpqc}})$, contradicting the assumption that $F_{\operatorname{fpqc}}$ is an fpqc sheaf. Note that this can be turned into an example on $\Sch^{\qcqs}$ by left Kan extending along ${\operatorname {Field}}_\QQ^{op} \to \Sch^{\qcqs}$.
\end{exam}

\subsection{Algebraic $K$-theory} %
By \cite{KST-Weibel}, non-connective $K$-theory satisfies procdh excision for Noetherian schemes. Consequently, by  Proposition~\ref{prop:excisionCech} it defines a sheaf $K\in \Shv_{\procdh}(\Sch^\noe,\Spt)$. We have the following topos-theoretic interpretation of the Bass construction.

\begin{theo}\label{thm;apcdhKconn=K}
For $X\in \Sch^\noe$ with $\dim(X)<\infty$, there exists a natural equivalence
\[ (a_{\procdh} τ_{\geq 0}K)(X) \simeq K(X).\]
\end{theo}

\begin{proof}
For $X$ as above, we have the descent spectral sequences
\[ E_2^{p,q} = H^p_{\procdh}(X,\tK_{-q}) \Rightarrow K_{-p-q}(X),\]
\[ E_2^{p,q} = H^p_{\procdh}(X,\conc\tK_{-q}) \Rightarrow \pi_{-p-q}(a_{\procdh} τ_{\geq 0}K(X)),\]
where $\tK_i$ is the procdh sheafification of the presheaf $K_i=\pi_i K$ of abelian groups on $\Sch_X$, and  $\conc\tK_i=\tK_i$ for $i\geq 0$ and $\conc\tK_i=0$ for $i< 0$.
By Corollary \ref{coro:finCohDim}, the spectral sequences are bounded.
So, it suffices to show that $\tK_i=0$ for $i<0$.
By Theorem~\ref{theo:procdhEnoughPoints} the procdh 1-topos has enough points, so it is enough to show that $\phi^*K_i = 0$ for all $i < 0$ and fibre functors $\phi^*: \Shv_{\procdh}(\Sch_X) \to \Set$. Since $K$-theory commutes with filtered colimits of rings, we have $\phi^*K_i = K_i(R)$ where $R = \colim R_λ$ for $(\Spec(R_λ) \to X)_{\lambda \in \Lambda}$ the proobject corresponding to $\phi^*$.
Let $\fN\subset R$ be the ideal of nilpotent elements.
Then, for $i<0$, we get $K_i(R) =  K_i(R/\fN) = 0$, where the first equality follow from the nil-invariance of the negative $K$-theory
and the last equality follows from \cite[Th.1.3]{KM} since $R/\fN$ is a valuation ring. This completes the proof.
\end{proof}

%
\begin{rema}\label{rmk;apcdhKconn=K}
%
The proof of Theorem~\ref{thm;apcdhKconn=K} shows the following. Note that we use Corollary~\ref{coro:smallFamily} to reduce to the smaller class of those procdh local rings $R$ with $\length Q(R)$ and $\dim R$ finite.

\begin{prop} \label{prop:procdhNonconnectification}
Take $\FF = \FF_p$ or $\QQ$, set $\cC = \Spt$ or $D(\ZZ)$, and let $\cE \in \PSh(\Sch^{\qcqs}_\FF, \cC)$ be a presheaf. Consider the following conditions.%
\begin{enumerate}
 \item[(\conDescent)] For each Noetherian $\FF$-scheme $X$ the restriction $\cE|_{\Sch_X}$ is a procdh sheaf.
 \item[(\conFinitary)] $\cE$ is finitary, in the sense that it preserves filtered colimits of $\FF$-algebras.
 \item[(\conBB)$_{\geq N}$] For each procdh local ring $R$ with $\length Q(R)$ and $\dim R$ finite, the procdh stalk $\cE(R)$ is homologically bounded below $N$. That is, $π_i \cE(R) = 0$ for $i < N$.
\end{enumerate}
If $\cE$ satisfies (\conDescent), (\conFinitary), and (\conBB)$_N$ then for every Noetherian $\FF$-scheme $X$ there is a natural equivalence
\begin{equation} \label{eq;rmk;apcdhKconn=K}
a_{\procdh} (F_{\geq N} )(X) \simeq F(X).
\end{equation}
\end{prop}
\end{rema}

We will produce a number of presheaves that satisfy the conditions of Proposition~\ref{prop:procdhNonconnectification} using the following definition.

\begin{defi}
Let $\FF$ and $\cC$ be as in Proposition~\ref{prop:procdhNonconnectification}. For $\cE \in \PSh(\Sch^{\qcqs}_\FF, \cC)$, define
\begin{equation} \label{NilF}
{ \Nil \cE:=\fib(\cE \to a_{\cdh} \cE)\in \PSh(\Sch_\FF^{\qcqs},\cC),}
\end{equation}
where $a_{\cdh}: \PSh(\Sch_\FF^{\qcqs},\cC)\to \Shv_{\cdh}(\Sch_\FF^{\qcqs},\cC)$ is the cdh sheafification functor.
\end{defi}

\subsection{Negative cyclic homology} %
Recall that for $k \to R$ a morphism of commutative rings, one defines $\HN(R/k) = HH(R/k)^{hS^1}$ as the homotopy fixed points of the Hochsdhild homology and for qcqs $k$-schemes using Zariski descent. In \cite[Thm.1.1]{An}, Antieau defines a functorial complete decreasing multiplicative $\ZZ$-indexed filtration,%
\footnote{It is not necessarily exhaustive in general, but is exhaustive if $X/k$ is $L_{X/k}$ has Tor-amplitude contained in $[0,1]$.}
\begin{equation} \label{HKR}
\biggl\{\FHKR^n \HN(X/k)\biggr\}_{n\in \ZZ}\;\text{on $\HN(X/k)$}
\end{equation}
and natural equivalences  
\begin{equation} \label{equa:grFHKR}
\gr_{\FHKR}^n \HN(X/k) \simeq \hLOH n {X/k}[2n].
\end{equation}
Here, $\hLO {X/\QQ}$ is the Hodge-completed derived de Rham complex equipped with the Hodge filtration $\bigl\{\hLOH n {X/k}\bigl\}_{n\in \NN}$.%
\footnote{That is,
$
\hLO {X/k} = \varprojlim_{n} \LO {X/k} ^{< n}
$ 
where for affines $X$, the complex $\LO {X/k} ^{< n}$ is the totalisation of the simplicial chain complex $[r] \mapsto \sigma^{< n}\Omega^*_{P_r/A}$ for $P_\bullet \to B$ a polynomial $A$-algebra resolution of $B$ and where $σ^{< n}$ refers to the stupid truncation in cohomological degrees $< n$.} %
The graded pieces of this filtration are computed by
\begin{equation} \label{equa:hLO}
gr^n \hLO {X/k}  \simeq \wedge^n L_{X/k}[-n]
\end{equation} 
where $L_{X/k}\in \PSh(\Sch^\qcqs_k,D(k)))$ is the cotangent complex, \cite[Construction~4.1]{Bha12a}.

By \cite[Lem.4.5]{EM}, 
the cofibre ${\operatorname{cofib}}(\FHKR^0\HN(-/\QQ) \to \HN(-/\QQ))$ is a cdh sheaf on $\Sch^\qcqs_\QQ$, so 
the canonical morphism
$\Nil\FHKR^0\HN(-/\QQ) \stackrel{\sim}{\to} \Nil\HN(-/\QQ) $
is an equivalence of presheaves. Therefore, applying $\Nil$ to \eqref{HKR} produces a complete and exhaustive $\NN$-indexed filtration 
\begin{equation} \label{NilHKR0}
{\biggl\{\FHKR^n \Nil \HN(-/\QQ)\biggl\}_{n\in \NN}\;\text{ on $\Nil\HN(-/\QQ)$}} 
\end{equation}
with identifications
\begin{equation} \label{NilHKR}
{ \gr^n_{\FHKR} \Nil \HN(-/\QQ) \simeq \Nil\hLOH n {-/\QQ}[2n].}
\end{equation}

\begin{lemm} \label{lemm:nilhLO}
The presheaves $\Nil \hLO {-/\QQ}^{\geq n} \in \PSh(\Sch^\qcqs_\QQ,D(\QQ))$ satisfy (\conDescent), (\conFinitary), and (\conBB)$_{\geq -n} $.
\end{lemm}

\begin{proof}\ 
\begin{enumerate}
 \item[(\conDescent)] By \cite[Thm.2.10]{Morrow}, the $\wedge^n L_{-/\QQ}\in \PSh(\Sch^\qcqs_\QQ,D(\QQ))$ are procdh sheaves on $\Sch_\QQ^{\noe}$. Hence, the  $L\Omega^{< n}_{-/\QQ}$, and their limit $\hLO {-/\QQ} = \lim_{n} L\Omega^{< n}_{-/\QQ}$ are also procdh sheaves. From this we deduce that the fibres $\hLO {-/\QQ} ^{\geq n} = \fib(\hLO {-/\QQ} \to \LO {-/\QQ} ^{< n})$ are procdh sheaves. Since all cdh sheaves are procdh sheaves, the fibres $\Nil\hLOH n {-/\QQ} = \fib(\hLO {-/\QQ} ^{\geq n} \to a_{\cdh}\hLO {-/\QQ} ^{\geq n})$ are procdh sheaves.

 \item[(\conFinitary)]	%
  Above we saw that $\hLO {-/\QQ}$ is a procdh sheaf. In fact, it is a cdh sheaf, \cite[Lem.4.5]{EM}, so $\Nil \hLO {-/\QQ} = 0$ leading to an equivalence 
  \begin{equation} \label{equa:nilhLOLO}
\Nil \hLO {-/\QQ}^{\geq n} \cong \Nil \LO {-/\QQ}^{<n}[-1].
\end{equation}
So it suffices to show finitarity of $\Nil \LO {-/\QQ}^{<n}$.
The presheaves $\wedge^i L_{-/\QQ}$ are finitary 
and $\LO {-/\QQ}^{<n}$ admits a finite filtration whose graded quotients are shifts of $\wedge^i L_{-/\QQ}$,
so it follows that the $\LO {-/\QQ}^{<n}$ are finitary. The cdh sheafification functor preserves finitary sheaves so the $a_{\cdh}\LO {-/\QQ}^{<n}$, and therefore $\Nil \LO {-/\QQ}^{<n}$ are finitary. %

 \item[{(\conBB)$_{\geq -n}$} ] By Eq.\eqref{equa:nilhLOLO} it suffices to show that $\LO {R/\QQ}^{<n}$ is supported in cohomological degree $\leq n-1$ for all $R$ as in Proposition~\ref{prop:procdhNonconnectification}.
We have 
\begin{equation}\label{eq1;cor;apcdhKconn=K}
{ (a_{\cdh}\LO{-/\QQ}^{<n})(R)= (a_{\cdh}\LO{-/\QQ}^{<n})(R/\fN) =\LO{(R/\fN)/\QQ}^{<n},}
\end{equation}
where the first (resp. second) equality holds since any finitary cdh sheaf is nil-invariant (resp. a valuation ring is a point of the cdh topos).
Since $\LO {A/\QQ}^{<n}$ for a local $\QQ$-algebra $A$ is supported in degrees $\leq n-1$, we are reduced to showing the surjectivity of the map
\[ H^{n-1}(\LO{R/\QQ}^{<n}) \to H^{n-1}(\LO{(R/\fN)/\QQ}^{<n}) .\]
This holds since the map is identified with
$\Omega^{n-1}_{R/\QQ}\to \Omega^{n-1}_{(R/\fN)/\QQ}$.
\end{enumerate}
\end{proof}

\begin{prop}\label{lemm:nilHN}
Let $X$ be a Noetherian $\QQ$-scheme of finite Krull dimension. Then $\Nil\HN(-/\QQ)\in \PSh(\Sch_{X}, D(\QQ))$ satisfy (\conDescent), (\conFinitary), and (\conBB)$_{\geq 0}$. Therefore 
\[a_{\procdh} (\Nil\HN(-/\QQ))_{\geq 0} (X) \simeq \Nil\HN(X/\QQ).\]
\end{prop}

\begin{proof}
The property (\conDescent) 
follows from the corresponding property for the graded pieces $\Nil \hLO {-/\QQ}^{\geq n}$ of the filtration \eqref{NilHKR0}. For (\conFinitary) and (\conBB)$_{\geq 0}$ consider the spectral sequence induced by \eqref{NilHKR0}
\[ E_2^{i,j}= H^{i-j}(\Nil\hLOH {-j} {X/\QQ}) \Rightarrow H^{i+j}\Nil\HN(X/\QQ).\]
If $X\in \Sch^\qcqs_{\QQ}$ has finite valuative dimension $d$ with Noetherian underlying topological spaces, Corollary~\ref{coro:finCohDim} and 
(\conBB)$_{\geq j}$ for $\Nil\hLOH {-j} {-/\QQ}$ from Lemma~\ref{lemm:nilhLO}
imply $E_2^{i,j}=0$ for $i-j>2d-j$, 
which implies that 
the spectral sequence is bounded. Hence, (\conFinitary) and (\conBB)$_{\geq 0}$ for $\HN(-/\QQ)$ follows from (\conFinitary) and (\conBB)$_{\geq j}$ 
for $\Nil\hLOH {-j} {-/\QQ}$, Lemma \ref{lemm:nilhLO}.
\end{proof}

\subsection{Integral topological cyclic homology}

For $X \in \Sch^{\qcqs}_{\FF_p}$ write $\TC(X)$ for the integral topological cyclic homology of $X$. By \cite{BMS2}, for any qcqs $\FF_p$-scheme $X$, there exists a functorial complete decreasing $\NN$-indexed filtration\footnote{\cite{BMS2} treats quasi-syntomic rings and it is extended to all $p$-complete rings in \cite{AMMN}.} 
\begin{equation} \label{eq;BMSfilt}
{\biggl\{ \FBMS^n \TC(X)\biggl\}_{n\in \NN} \text{ on } \TC(X)} 
\end{equation}
with associated graded quotients 
\[\gr_{\FBMS}^n \TC(X) \simeq \Znsyn(X)[2n]\]
for a natural object $\Znsyn\in \PSh(\Sch_{\FF_p},D(\ZZ_p))$ called the syntomic complex. 

\begin{lemm} \label{prop:NilZnsyn}
The presheaf $\Nil\Znsyn \in \PSh(\Sch^{\qcqs}_{\FF_p}, \Spt)$ satisfies (\conDescent), (\conFinitary) and (\conBB)$_{\geq -n}$.
\end{lemm}

\begin{proof}\ 
\begin{enumerate}
 \item[(\conDescent)] This property for $\Nil\Znsyn$ follows from the fact that  $\Znsyn$ is a procdh sheaf on $\Sch^\noe_{\FF_p}$ by \cite[Th.8.6]{EM} and Proposition \ref{prop:excisionCech}.

 \item[(\conFinitary)] The argument is taken from the proof of \cite[Th.4.24(4)]{EM}.
First, we claim $\Nil\Znsyn[\frac{1}{p}]=0$. Using the fact that $\FBMS^\bullet\TC$ from \eqref{eq;BMSfilt} naturally splits after inverting $p$, the claim is reduced to 
$\Nil\TC[\frac{1}{p}]=0$. To prove this, we use the following pullback square in $\PSh(\Sch^\qcqs,\Spt)$
\begin{equation} \label{eq;Kinfsquare}
{\xymatrix{
K \ar[d] \ar[r]^-{\tr}  & \TC \ar[d] \\
KH \ar[r]^-{\tr^{cdh}}& a_{\cdh} \TC\\},}
\end{equation}
where $KH$ is the homotopy $K$-theory and $\tr$ is the cyclotomic trace (\cite{BHM}, \cite{DGM}), and $\tr^{cdh}$ is induced by $\tr$ via the equivalence $KH \simeq a_{\cdh} K$, \cite[Th.6.3]{KST-Weibel}, \cite{Cis13}, \cite{KM}.

The pullback square follows from the latter equivalence and the fact that 
the fiber of $\tr$ is a cdh sheaf by \cite[Th. A.3]{LandTamme}.
By \eqref{eq;Kinfsquare}, $\Nil\TC[\frac{1}{p}]=0$ follows from 
$\fib(K\to KH)[\frac{1}{p}]=0$ by \cite[Th. 9.6]{TT90}.
Thus, it suffices to show that $\Znsyn(-)/p$ is finitary noting that the cdh sheafification of a finitary presheaf is finitary.
By \cite[Lem.4.16]{EM}, $\Znsyn(-)/p$ admits a finite increasing filtration whose graded pieces are some shifts of $\wedge^i L\Omega_{-/\FF_p}$ with $i\leq n$,
so the desired assertion follows from the finitarity of $\wedge^i L\Omega_{-/\FF_p}$.

 \item[ {(\conBB)$_{\geq -n}$} ] By the finitarity and Theorem~\ref{theo:colimfgNilredhvr}, it suffices to show that 
$\Nil\Znsyn(R)\in D(\ZZ_p))^{\leq n}$ for any henselian local ring $R$ such that the ideal $\fN\subset R$ of nilpotent elements is finitely generated and $R/\fN$ is a valuation ring. 
Similarly to \eqref{eq1;cor;apcdhKconn=K}, we have 
\[a_{\cdh} \Znsyn(R)=a_{\cdh} \Znsyn (R/\fN)=\Znsyn(R/\fN).\]
Note that $\Znsyn$ and $a_{\cdh} \Znsyn$ are not finitary in general so the first equality requires the assumption that $\fN$ is finitely generated.
Hence,  {(\conBB)$_{\geq -n}$} follows from \cite[Th.5.2]{AMMN} noting $(R,\fN)$ is a henselian pair. This completes the proof of the claim.
\end{enumerate}
\end{proof}

\begin{prop} \label{theo:apcdhTC=TC}
The presheaf $\Nil \TC \in \PSh(\Sch_{\FF_p}^{\qcqs}, \Spt)$ satisfies (\conDescent), (\conFinitary), and (\conBB)$_{\geq 0}$. Consequently, for any Noetherian $\FF_p$-scheme $X$ with $\dim(X) < \infty$ we have 
\[a_{\procdh} (\Nil\TC)_{\geq 0} (X) \simeq \Nil\TC(X).\]
\end{prop}

\begin{proof}
The desired properties follow from Lemma~\ref{prop:NilZnsyn}
by the same argument deducing Proposition \ref{lemm:nilHN} from Lemma \ref{lemm:nilhLO}: The property (\conDescent) follows from 
(\conDescent) for $\Nil\Znsyn$ by the filtration \eqref{eq;BMSfilt}.
The properties (Fin) and (BB)$_{\geq 0}$ for $\Nil \TC$ follow from (Fin) and (BB)$_{\geq j}$ for $\Nil\ZZ_p(-j)^{syn}$ 
 by using the spectral sequence 
\[ E_2^{i,j}= H^{i-j}(\Nil\ZZ_p(-j)^{syn}(X)) \Rightarrow \pi_{-i-j}\Nil\TC(X),\]
arising from the filtration \eqref{eq;BMSfilt}.
Note that the spectral sequence is bounded since
Corollary~\ref{coro:finCohDim} and (\conBB)$_{\geq j}$ for $\Nil\ZZ_p(-j)^{syn}$ from Lemma~\ref{prop:NilZnsyn} imply $E_2^{i,j}=0$ for $i-j>2d-j$. 
\end{proof}

\section{Procdh motivic complex}

In this section, we address the following conjecture.

\begin{conj}[{Beilinson, \cite{Bei}, cf. also \cite[Introduction]{EM}}]\label{conj;Beilinson}
For any reasonable scheme $X$, there is a natural spectral sequence:
\begin{equation}\label{eq.AH}
E^{p,q}_{2}=H^{p-q}_{\mathcal{M}}(X, \ZZ(-q)) \Rightarrow K_{-p-q}(X)\end{equation}
where $K_*(X)$ is the non-connective algebraic $K$-theory of $X$, and  
$H^i_{\mathcal{M}}(X, \ZZ(n))$ is the \emph{motivic cohomology} of $X$ yet to be defined.
\end{conj}

When $X$ is smooth over a field $k$, an answer was given by the following theorem.

\begin{theo}[\cite{FrSus} and \cite{Lev08}]\label{thm;smooth}
Let $\Sm_k$ denote the category of smooth schemes separated of finite type over 
a field $k$. 
There exists a complete multiplicative decreasing $\NN$-indexed filtration $\biggl\{\Fmot^n K(X)\biggl\}_{n\in \NN}$ on $K(X)$ and identifications of spectra
\begin{equation} \label{equa:grKZn}
\gr_{\Fmot}^n K(X) \simeq \Znsm(X)[2n]
\end{equation}
functorial in $X\in \Sm_k$.
\end{theo}

Above, $\Znsm(X)[2n]$ is a chain complex of abelian groups regarded as a spectrum via the Eilenberg-Maclane functor $D(\ZZ) \to \Spt$. As a complex of abelian groups, it can be defined as 
\begin{equation} \label{equa:ZZtrGm}
\Znsm(X) = 
\underline{C}_*(\ZZ_{tr}(\GG_m^{\wedge n}))(X)[-n] \qfor X\in \Sm_k,
\end{equation}
where $\underline{C}_*(\ZZ_{tr}(\GG_m^{\wedge q}))[-q]$ is Voevodsky's $\AA^1$-invariant motivic complex defined in \cite{SVBK}. This is strictly functorial in $\Sm_k$ in the sense that it defines a functor between the 1-category $\Sm_k$ and the 1-category of chain complexes of abelian groups. Scheme-wise this is shown to be quasi-isomorphic to Bloch's cycle complex, \cite{Bl}, in \cite[Cor.2]{Voe02}. Of course, another approach is to just take Eq.\eqref{equa:grKZn} as the definition for $\Znsm(X)[2n]$.


\medbreak

In what follows, we write $\FF=\QQ$ or $\FF_p$ and continue to let $\Sch^\qcqs_\FF$ be the category of qcqs schemes over $\FF$ and $\Sch^\noe_\FF$ be its full subcategory of noetherian schemes.
Recently, Elmanto-Morrow \cite{EM} extended Theorem \ref{thm;smooth} to all 
$X\in \Sch^\qcqs_\FF$ by using instead of $\Znsm$ a new motivic complex
\[ \Znem\in \PSh(\Sch^\qcqs_\FF,D(\ZZ))).\]
They construct $\Znem$ by modifying the cdh sheafification $\Zncdh$ of the left Kan extension of $\Znsm$ along $\Sm_\FF \to \Sch^\qcqs_\FF$ by using Hodge-completed derived de Rham complexes in case $\FF=\QQ$ and syntomic complexes in case $\FF=\FF_p$.
The construction is motivated by trace methods in algebraic $K$-theory using the cyclotomic trace map $\tr$ from \eqref{eq;Kinfsquare}.
The purpose of this section is to give a different approach to Conjecture \ref{conj;Beilinson} by using our procdh topology.

\begin{defi}\label{def;Znpcdh}
For integers $n\geq 0$, we define the \emph{procdh-local motivic complex}
\[\Znpcdh:=a_{\procdh} \Lsm \Znsm \in \Shv_{\procdh}(\Sch^\qcqs_\FF,D\ZZ)),\]
as the procdh sheafification of the left Kan extension $\Lsm \Znsm$ of $\Znsm$ along $\Sm_\FF\to \Sch^\qcqs_\FF$.
\end{defi}

\begin{rema}
Note that there have been various constructions of motivic cohomology on smooth schemes over Dedekind domains, \cite{Gei04}, \cite{Lev01}, \cite{Spi18}, \cite{CD19}, etc.,
and one could apply the same construction, i.e., procdh sheafification of left Kan extension, to any of these.
\end{rema}

\begin{theo}\label{thm;AH-Kpcdh}
For $X\in \Sch^\noe_\FF$ with $\dim(X)<\infty$, there exists a 
complete multiplicative%
\footnote{For compatibility of the monoidal structures with left Kan extension see \cite[\S 2.3]{EM}. Sheafification is also compatible since it is defined using filtered colimits, see \cite[Prop.6.2.2.7]{HTT} which works for a very general class of coefficient categories.}
decreasing $\NN$-indexed filtration $\biggl\{\Fpcdh^n K(X)\biggl\}_{n\in \NN}$ on $K(X)$ and identifications 
\[ \gr_{\Fpcdh}^n K(X) \simeq \Znpcdh(X)[2n].\]
\end{theo}

\begin{proof}
By Bhatt-Lurie (see \cite[Ex. A.0.6]{EHKSY}),
there is a natural equivalence
\[ K_{\geq 0} \simeq \Lsm (K_{|\Sm_\FF}),\]
where the right hand side is the left Kan extension of  $K_{|\Sm_\FF}$ along $\Sm_\FF\to \Sch^\qcqs_\FF$. 
Using Theorem \ref{thm;apcdhKconn=K}, we obtain a filtration on $K$ by procdh sheafifying the left Kan extension along $\Sm_\FF\to \Sch_\FF$ of $\Fmot^\bullet K_{|\Sm_\FF}$ from Theorem \ref{thm;smooth}. The identification of the graded pieces is clear. Completeness follows from the boundedness described in Eq.\eqref{equa:FpcdhLbound} below.
\end{proof}

\begin{rema}\label{rmk;thm;AH-Kpcdh}
By definition, Eq.\eqref{equa:ZZtrGm},
we have 
$\mathcal{H}_{\Zar}^i(\Znsm)=0$ for $i>n$, where $\mathcal{H}_{\Zar}^i(\Znsm)$ is the Zariski cohomology sheaf of $\Znsm$. It implies 
that for any local $k$-algebra $A$, we have
\begin{equation} \label{LZnsm-vanishing}
{H^i((\Lsm\Znsm)(A))=0\qfor i>n.} 
\end{equation}
Thus, Corollary \ref{coro:finCohDim} implies that
for $X\in \Sch^\qcqs_\FF$ of finite valuative dimension $d$ with Noetherian underlying topological space, we have 
\begin{equation} \label{Znpcdh-vanishing}
{
H^i(\Znpcdh(X))=0\qfor i> 2d+n.
}
\end{equation}
In particular, when $X$ is noetherian, we have for each $n\in \NN$
\begin{equation} \label{equa:FpcdhLbound}
 π_{-i}\Fpcdh^n K(X) = 0;  \qquad i > 2d -n 
 \end{equation}
 so the induced spectral sequence \eqref{eq.AH} with
\[H^i_{\mathcal{M}}(X, \ZZ(n)):=H^i(\Znpcdh(X))\] is bounded.
\end{rema}

\medbreak

Now, it is a natural question if two constructions $\Znem$ and $\Znpcdh$ coincide.

\begin{theo}\label{thm;Znpcdh-comparison}
Assume given $\cZ(n)\in \PSh(\Sch_\FF^{\qcqs},D(\ZZ)))$ and consider the following conditions. 
\begin{enumerate}
\item[(a)] \label{thm:enum:compa}
There is a natural comparison map 
$$\psi: \Znsm \to \cZ(n)_{|\Sm_\FF}$$
 in $\PSh(\Sm_\FF,D(\ZZ)))$, whose induced map $\phi:\Lsm\Znsm \to \cZ(n)$ has the properties that $\phi(R)$ is an equivalence for all procdh local rings $R$ with $\length Q(R)$ and $\dim R$ finite.

\item[(b)]
$\cZ(n)$ is finitary. That is, 
$$\cZ(n)(\lim_{\lambda} P_λ) = \underset{\lambda}{\colim} \cZ(n)(P_λ)$$
 for any cofiltered system $(P_λ)_{λ \in \Lambda}$ in $\Sch^{\qcqs}_\FF$ with affine transition morphisms.

\item[(c)] $\cZ(n)$ is a procdh sheaf on Noetherian schemes. That is,
$$\cZ(n)\in \Shv_{\procdh}(\Sch^\noe_\FF,D(\ZZ))). $$
%
\end{enumerate}
If (a), (b), and (c) are satisfied then, $\phi$ induces an equivalence $\Znpcdh(X) \simeq \cZ(n)(X)$ for any $X\in \Sch^\noe_\FF$.
\end{theo}

\begin{proof}
For procdh sheaves of chain complexes $\cE \in \Shv_{\procdh}(\Sch^{\noe}_{\FF}, D(\ZZ))$, we have a descent spectral sequence 
\[ E_2^{p,q} = H^p_{\procdh}(X,\mathcal{H}_{\procdh}^q\cE) \Rightarrow H^{p+q}\cE(X),\]
which converges by finiteness of cohomological dimension, Cor.\ref{coro:finCohDim}. Here $\mathcal{H}_{\procdh}^q\cE$ is the procdh sheafification of the presheaf of abelian groups $X \mapsto π_{-q}\cE(X)$. So since $\cZ(n)$ and $\ZZ(n)^{\procdh}$ are procdh sheaves, Assumption~(c), it suffices to show that the morphism 
\[ 
\mathcal{H}_{\procdh}^q\Znpcdh
\to 
\mathcal{H}_{\procdh}^q\cZ(n)
\]
of procdh cohomology sheaves of abelian groups is an isomorphism on Noetherian schemes. %
Since the 1-topos $\Shv_{\procdh}(\Sch_X)$ has enough points, Thm.\ref{theo:procdhEnoughPoints},
and cohomology commutes with filtered colimits, 
and both $\cZ(n)$ and $\Lsm\Znsm$ are finitary, Assumption~(b),  
it suffices to show that for all procdh local rings $\Spec(R) \to X$ with $R$ as in Assumption~(a), the comparison 
\[ 
H^q\Lsm\Znsm(R) 
 \to 
 H^q\cZ(n)(R) 
 \] 
is an isomorphism for all $q \in \ZZ$.
This is precisely Assumption~(a).
\end{proof}

\begin{theo}[Elmanto-Morrow]\label{thm;Znem-comparison}
The presheaf $\Znem$ satisfies the conditions (a), (b), (c) of Theorem \ref{thm;Znpcdh-comparison}.
\end{theo}

\begin{proof}
Condition (b), i.e., finitary-ness, is \cite[Th.4.10(5), Th.4.24(4)]{EM} and 
Condition (c), i.e., procdh descent, is \cite[Th.8.2]{EM}.

We prove (a). Existence of the morphism $\Znsm \to \cZ(n)_{|\Sm_\FF}$ is contained in \cite[Th.1.1(9)]{EM}, although we don't need the full force of their theorem since we are only asking for existence, not the stronger fact that the morphism is an equivalence. Now we want to show that the induced map $\Lsm\Znsm \to \cZ(n)$ is an equivalence on procdh local rings $R$ with $\length Q(R)$ and $\dim R$ finite.

 By \cite[Th.7.7]{EM}, it suffices to prove $\Znem(R)$ is supported in cohomological degrees $\leq n$ for any procdh local ring $R$.
By \cite[Th.4.10(2), Th.4.24.(2)]{EM}, there are fiber sequences (cf.  Eq.\eqref{NilF})
\[ \Nil\hLOH n {R/\QQ} \to  \Znem(R) \to \Zncdh(R)\;\text{ if } \FF=\QQ,\]
\[ \Nil\Znsyn(R) \to  \Znem(R) \to \Zncdh(R)\;\text{ if } \FF=\FF_p,\]
where $\Zncdh=a_{\cdh}\Lsm\Znsm$.
Note that $\Nil\hLOH n {R/\QQ}$ and $\Nil\Znsyn(R) $ are supported in degrees $\leq n$ by Lemma~\ref{lemm:nilhLO} and Lemma~\ref{prop:NilZnsyn}.
Noting that $\Znsm$ is finitary, the same argument as Lemma~\ref{prop:NilZnsyn} gives 
\[\Zncdh(R)=\Zncdh(R/\fN)=(\Lsm\Znsm)(R/\fN).\]
So, $\Zncdh(R)$ is supported in degrees $\leq n$ by Eq.\eqref{LZnsm-vanishing}, which  proves the desired assertion.
\end{proof}

\begin{coro}\label{cor;Znem-comparison}
There is an equivalence $\Znpcdh \simeq \Znem$ in $\PSh(\Sch^\noe_\FF,D(\ZZ)))$. 
 \end{coro}

\def\tnu{\tilde{\nu}}

\bibliographystyle{amsalpha}
\bibliography{bib}

\end{document}